   \documentclass{amsart}

\usepackage{amssymb, amsfonts,amssymb, hyperref, amsmath,graphicx,graphics,color}

\newcommand{\pa}{\partial}

\newcommand{\bp}{\textbf{p}}

\newcommand{\bx}{{\bf x}}

\newcommand{\bz}{{\bf z}}

\newcommand{\n}{\nabla}

\newcommand{\N}{\mathbb{N}}
\newcommand{\M}{\mathfrak{M}}
\newcommand{\R}{\mathbb{R}}

\newcommand{\Z}{\mathbb{Z}}

\newcommand{\cC}{\mathcal{C}}
\newcommand{\cL}{\mathcal{L}}
\renewcommand{\epsilon}{\varepsilon}
\newcommand{\be}{\textbf{e}}

\renewcommand{\qed}{\hspace*{\fill} \setlength{\unitlength}{1mm}
\begin{picture}(2.5,2.5)
      \put(0,0){\framebox(2.5,2.5){}}
\end{picture}
\setlength{\unitlength}{1pt}}

        \definecolor{pink}{rgb}{1,0,1}

\begin{document} 

\title{The fundamental gap of simplices}

\author[Zhiqin Lu]{Zhiqin Lu}
\address{Department of Mathematics, University of California, Irvine, CA 92697-3875.} \email{zlu@uci.edu}

\author[Julie Rowlett]{Julie Rowlett}
\address{Max Planck Institut f\"ur Mathematik,  Vivatgasse 7,  D-53111 Bonn, Germany.}  \email{rowlett@math.uni-bonn.de}

\newtheorem{theorem}{Theorem}
\newtheorem{proposition}{Proposition}
\newtheorem{lemma}{Lemma}
\newtheorem{cor}{Corollary}
\newtheorem{question}{Question}
\newtheorem{claim}{Claim}
\newtheorem{note}{Note}
\newtheorem{conjecture}{Conjecture}
\theoremstyle{remark}
\newtheorem{remark}{Remark}
\theoremstyle{definition}
\newtheorem{defn}[equation]{Definition}
\newtheorem*{theorem1}{Theorem 1}
\newtheorem*{lemma1}{Lemma 1}
\newtheorem*{cor1}{Corollary 1}
\newtheorem*{theorem2}{Theorem 2}

%

%
%

\begin{abstract}
The gap function of a domain $\Omega \subset \R^n$ is 
$$\xi(\Omega) := d^2 (\lambda_2 - \lambda_1),$$
where $d$ is the diameter of $\Omega$, and $\lambda_1$ and $\lambda_2$ are the first two positive Dirichlet eigenvalues of the Euclidean Laplacian on $\Omega$.  It was recently shown by Andrews and Clutterbuck \cite{ac} that for any convex $\Omega \subset \R^n$, 
$$\xi(\Omega) \geq 3 \pi^2,$$
where the infimum occurs for $n=1$.  On the other hand, the gap function on the moduli space of $n$-simplices behaves differently.  Our first theorem is a compactness result for the gap function on the moduli space of $n$-simplices.  Next, specializing to $n=2$, our second main result proves the recent conjecture of Antunes-Freitas \cite{af}:  \em for any triangle $T \subset \R^2$, 
$$\xi(T) \geq \frac{64 \pi^2}{9},$$
with equality if and only if $T$ is equilateral.  \em  
\end{abstract}

\maketitle
\section{Introduction}
Let $\Omega \subset \R^n$ be a convex domain.  Let $\lambda_1$ and $\lambda_2$ be the first two eigenvalues of the Euclidean Laplacian on $\Omega$ with Dirichlet boundary condition.  It is a classical result that $0 < \lambda_1 < \lambda_2$.  The \em gap \em between $\lambda_1$ and the rest of the spectrum, 
$$\lambda_2 - \lambda_1$$
is known as the \em fundamental gap of $\Omega$.  \em  The \em gap function \em 
$$\xi(\Omega) = d^2 (\lambda_2 - \lambda_1),$$
where $d$ is the diameter of $\Omega$.  The gap function is a scale invariant:  it is purely determined by the shape of the domain.  Physically, if we consider heating the domain at some initial time and then keeping the boundary of the domain fixed at zero temperature, the fundamental gap determines the rate at which the overall heat in the domain vanishes as time tends to infinity.  It is natural to ask the following question.  
\begin{quote} 
How does the shape of a convex domain affect the rate at which it loses heat over a long period of time?   
\end{quote}
The mathematical formulation of this question is:
\begin{quote}
 What is the relationship between the geometry of a convex domain $\Omega \subset \R^n$ and $\xi(\Omega)$?  
\end{quote}

M. van den Berg \cite{mvdb} observed that for many convex domains, the gap function is bounded below by a constant.  For example, consider a rectangular domain $R \subset \R^2$, 
$$R = \left\{ (x, y) \in \R^2 \, \mid \, 0 \leq x \leq a, \, 0 \leq y \leq b \right\}.$$
Using separation of variables, it is straightforward to compute that the eigenfunctions and corresponding eigenvalues of the rectangle are 
$$\phi_{j,k} (x,y) = \sin\left( \frac{jx\pi}{a} \right) \sin\left(\frac{ky\pi}{b} \right), \quad \lambda_{j,k} = \frac{j^2 \pi^2}{a^2} + \frac{k^2 \pi^2}{b^2}, \quad j,k \in \N.$$
Making the additional assumption $b \leq a$, one computes the gap function of the rectangle $R$, 
$$\xi(R)= \frac{3\pi^2(a^2 + b^2)}{a^2}.$$
If we then think about the gap function on all possible rectangles $R$, we see that the square uniquely maximizes the gap function with 
$$\xi(\textrm{Square}) = 6 \pi^2.$$
On the other hand, if a rectangle collapses to a segment, by letting $b \downarrow 0$, then $\xi \downarrow 3 \pi^2$.  An even more elementary example is the segment.  The gap function on any (finite) segment $[a,b]$ with $a<b$ is 
$$3 \pi^2.$$
Perhaps based on this intuition, Yau formulated the \em fundamental gap conjecture \em in \cite{yau} which was recently proven by Andrews and Clutterbuck \cite{ac}.
\begin{theorem}[Andrews-Clutterbuck] 
 For any convex domain in $\R^n$, the gap function is bounded below by $3\pi^2$.  
\end{theorem}    

This result shows that among \em all convex domains, \em the gap function is minimized in dimension $1$.  If the gap function is restricted to a certain moduli space of convex domains, what are its properties?  

In this work, we focus on the gap function \em restricted to the moduli space of $n$-simplices \em and in particular, the \em moduli space of Euclidean triangles.  \em     Recall that an $n$-simplex $X$ is a set of $n+1$ vectors $\{v_0,\cdots,v_n\}$ in $\mathbb R^n$ such that
$v_1-v_0,\cdots,v_n-v_0$ are linearly independent. The convex domain
\[
\left\{\left . \sum_{j=0}^n t_j v_j \right| \sum_{j=0} ^n t_j=1, \,  t_j\geq 0
\textrm{ for } 0 \leq j \leq n \right\}
\]
defined by $X$ is bounded with piecewise smooth boundary.  For the
sake of simplicity, we don't distinguish the simplex $X$ with the
domain it defines.  The moduli space of $n$-simplices is the set of all similarity classes of $n$-simplices; it is parametrized by the set of $n$-simplices with diameter equal to one.  We note that in case $n=2$, this theorem is straightforward to deduce from the main result of \cite{strip}.  

\begin{theorem} \label{th:simplex}
Let $Y$ be an $n-1$ simplex for some $n \geq 2$.  Let $\{X_j \}_{j \in
\N}$ be a sequence of $n$ simplices each of which is a graph over $Y$.
 Assume the height of $X_j$ over $Y$ vanishes as $j \to \infty$.  Then
$\xi(X_j) \to \infty$ as $j \to \infty$.  More precisely, there is a
constant $C>0$ depending only on $n$ and $Y$ such that
$\xi(X_j)\geq C h(X_j)^{-4/3}$, where $h(X_j)$ is the height of $X_j$.
\end{theorem}
 
Since any triangle with unit diameter is a graph over the unit interval, this theorem implies that \em there exists at least one triangle which minimizes the gap function on the moduli space of triangles.  \em   The \em moduli space of triangles \em is the set of all similarity classes of triangles, which we identify with 
$$\mathfrak{M} \cong \left\{ (x,y) \in \R^2 : y>0, \quad \frac{1}{2} \leq x \leq 1, \quad x^2 + y^2 \leq 1\right\},$$
where the vertices of a triangle in each similarity class are $(0,0)$, $(1, 0)$ and $(x,y)$.  The following result shows that the gap function on triangular domains is more than twice as large as the gap function on a generic convex domain; the theorem was conjectured in \cite{af}.  

\begin{theorem}\label{th:emin}  
For any triangle $T$ with unit diameter, 
$$\xi(T) \geq \frac{64 \pi^2}{9},$$
where equality holds iff $T$ is equilateral. 
\end{theorem}  

Let us recall the famous question posed by M. Kac \cite{kac}:
\begin{quote}
Can one hear the shape of a drum?
\end{quote}
 The resonant tones of a domain are in bijection with the eigenvalues of the Euclidean Laplacian with Dirichlet boundary condition.  Therefore, with a \em perfect ear \em that is capable of registering all tones, one can \em hear \em the spectrum, that is, the set of all eigenvalues.  Kac's question is then mathematically reformulated as follows.  
\begin{quote}
 If two domains in $\R^2$ have the same spectrum, do the domains also have the same shape?
\end{quote}
A negative answer to Kac's question was demonstrated by Gordon, Webb and Wolpert \cite{gww}, who showed that there exist isospectral planar domains which are not isomorphic.  On the other hand, Durso \cite{durso} proved that if the two domains are triangles in $\R^2$, and they have the same spectrum, then they must be the same triangle.  The proof uses \em the entire spectrum, \em so we can reformulate her result as follows.
\begin{quote}
With a perfect ear, one can hear the shape of a triangle \cite{durso}.  
\end{quote} 
In practice, however, one does not have a perfect ear.  That is, one may only detect a finite portion of the spectrum.  Our Theorem  \ref{th:emin} implies that the equilateral triangle can be heard with a \em realistic ear, \em because Theorem \ref{th:emin} demonstrates that the gap function alone uniquely distinguishes the equilateral triangle within the moduli space of all triangles.  In fact, we expect that it is possible to distinguish triangles based on a finite number of eigenvalues.  This is supported by numerical data in \cite{af-2},  which shows that one expects that triangles are uniquely determined by their first three eigenvalues.  

Our work is organized as follows.  The compactness result for simplices is proven in \S 2; in \S 3 this is refined to prove that Theorem \ref{th:emin} holds for all sufficiently ``thin'' triangles.  In \S 4, we prove that the equilateral triangle is a \em strict local minimum \em for the gap function on the moduli space of triangles, and in \S 5 we determine a lower bound for the radius of the neighborhood in the moduli space of triangles on which the equilateral triangle is a strict local minimum.  Finally, in \S 6, we provide an algorithm to complete the proof of Theorem \ref{th:emin}.  Concluding remarks and conjectures are offered in \S 7.    

\section*{Acknowledgements}  
We are deeply grateful to Timo Betcke for assisting our numerical calculations in \S 6.  The first author is supported by NSF grant DMS-12-06748, and the second author acknowledges the support of the Max Planck Institut f\"ur Mathematik in Bonn.  Both authors would like to thank Gilles Carron and Pedro Freitas for insightful comments.  

\section{A compactness result for the gap of simplices} 
Let us first fix the notation.  The Laplace operator on $\R^n$ is 
$$\Delta = \sum_{i=1} ^n \frac{\pa^2}{\pa x_i ^2},$$
The Dirichlet (respectively, Neumann) eigenvalues of the Laplace operator are the real numbers $\lambda$ for which there exists an eigenfunction
$$u \in \cC^{\infty} (\Omega) \textrm{ such that } -\Delta u = \lambda u \textrm{ and } \left. u \right|_{\pa \Omega} = 0, \textrm{ (respectively, } \left. \frac{\pa u}{\pa n} \right|_{\pa \Omega} =0).$$

We shall use $\lambda$ to denote Dirichlet eigenvalues, $\mu$ to denote Neumann eigenvalues, and index the Dirichlet eigenvalues by $\N$ and the Neumann eigenvalues by $0 \cup \N$.  The Dirichlet and Neumann\footnote{Note that the Neumann boundary condition is automatically satisfied if no boundary condition is imposed in the variational principle.}  eigenvalues, respectively, satisfy the following variational principles \cite{chavel},
\begin{equation} \label{varprin} \lambda_k = \inf_{\varphi \in \cC^1 (M)} \left\{ \left. \frac{ \int_{M} |\nabla \varphi |^2 }{\int_{M} f^2 } \, \right| \, \left. \,  \varphi \right|_{\pa M} = 0, \, \varphi \not\equiv 0 = \int_{M} \varphi \phi_j, \, 0 \leq  j < k \right\},\end{equation} 
$$\mu_j = \inf_{\varphi \in \cC^1 (M)} \left\{   \left. \frac{ \int_{M} |\nabla \varphi |^2 }{\int_{M} \varphi^2 } \, \right| \, \,  \varphi \not\equiv 0 = \int_{M} \varphi \varphi_l, \, -1 \leq l < j \right\},$$
for $k\geq 1$ and $j\geq 0$ where $\phi_j$ and $\varphi_l$ are, respectively, eigenfunctions for $\lambda_j$ and $\mu_l$ (assuming that
$\phi_0\equiv 0$ and $\varphi_{-1}\equiv 0$).
The well known property of \em domain monotonicity \em for the Dirichlet eigenvalues is that, if a domain $\Omega \subset \Omega^*$, then 
$$\lambda_k (\Omega) \geq \lambda_k (\Omega^*).$$

In \cite{lr}, we demonstrated the following \em weighted variational principle \em for so-called \em drift Laplacians.  \em  Given a \em weight function \em $\phi$, the drift Laplacian with weight $\phi$ is 
$$\Delta_\phi := \Delta - \n \phi \cdot \n.$$
The Dirichlet and Neumann eigenvalues of the drift Laplacian satisfy the following variational principles.  
$$\lambda_k =  \inf_{\varphi \in \cC^1(M)}  \left \{ \left . \frac{\int_M|\nabla\varphi|^2 e^{-\phi} }{\int_M\varphi^2 e^{-\phi}} \, \right|  \, \,  \varphi \not\equiv 0 = \int_{M} \varphi \phi_j e^{-\phi}, \, \, 0 \leq j < k, \, \varphi |_{\pa M} = 0 \right\},
$$
$$\mu_k =  \inf_{\varphi \in \cC^1(M)}  \left \{ \left . \frac{\int_M|\nabla\varphi|^2 e^{-\phi} }{\int_M\varphi^2 e^{-\phi}} \, \right|  \, \,  \varphi \not\equiv 0 = \int_{M} \varphi \varphi_j e^{-\phi}, \, -1 \leq j < k \right\}
$$
for $k\geq 1$ and $l\geq 0$, where $\varphi_j$ achieves the infimum for $k=j$ (and as above
$\phi_0\equiv 0$ and $\varphi_{-1}\equiv 0$).  Finally, throughout this paper we will use the following notations:  for a function $f(t)$ and fixed $k \geq 0$, 
$$f(t) = O(t^k) \textrm{ as } t \to 0 \textrm{ if there exist } C, \delta > 0 \textrm{ such that } |f(t)| \leq Ct^k \textrm{ for all } |t| \leq \delta;$$
$$f(t) = o(t^k) \textrm{ as } t \to 0 \textrm{ if } \lim_{t \to 0} \frac{f(t)}{t^k} = 0.$$   

\subsection{Proof of Theorem \ref{th:simplex}}  
To prove Theorem \ref{th:simplex}, we show that if a sequence of $n$-simplices collapse, there exists $C>0$ such that $\xi(X_j) \geq C h(X_j)^{-4/3}$, where the height $h$ of the simplex (defined in the arguments below) vanishes as $j \to \infty$.  For simplicity in notation, let us drop the subscript.  We may assume that the simplex is defined by the points 
$$\{ \bp^j \}_{j=0} ^n \subset \R^n, \quad \bp^0 = 0,$$
such that 
$$\bp^j = \sum_{i=1} ^n p^j _i \be^i, \quad p^j _n= 0, \quad 1 \leq j \leq n-1,$$
where $\{ \be^i \}_{i=1} ^n$ are the standard basis of $\R^n$.  In other words, $\bp^0, \ldots, \bp^{n-1}$ are contained in the span of $\{ \be^i\}_{i=1} ^{n-1}$.  The collapse is described by 
$$|p_n ^n | \to 0.$$
In fact, we may assume without loss of generality that the simplex is contained in the set of points
$$\left\{ \left. \bx \in \R^n \right| \, \,  \bx = \sum_{k=1} ^n x_k \be^k, \quad x_n \geq 0\right\}.$$
Then, for any point 
$$\bx \in X, \quad \bx = \sum_{k=1} ^{n} x_k \be^k,$$
the \em height \em of $\bx$ 
$$h(\bx) := x_n.$$
The height of the simplex itself is defined to be 
$$h = h(X) := h(\bp^n) = p_n ^n.$$
Since the simplex collapses, we assume in the remaining arguments that $h < 0.1$.  

Let $\lambda_i$, $i=1,2$, be the first and second Dirichlet eigenvalues of $X$ with corresponding eigenfunctions $\phi_i$ such that $\int_X \phi_i ^2 = 1$.  In the following claim, we demonstrate the quantitative estimate that at least 90 \% of the mass of the eigenfunctions $\phi_1$ and $\phi_2$ is contained in a cylinder around $\bp^n$ intersected with $X$.   We call this estimate ``cutting corners'' because it shows that we may ``cut off the corners'' and use the cylinder to estimate the gap.  
Let
\[
\tilde\bp :=\sum_{i=1}^{n-1} p^n_i\be ^i, \]
and let  \[ B_{ch^{2/3}} ^{n-1} (\tilde\bp)\] 
be the $(n-1)$ dimensional ball in the space spanned by $\be^1,\cdots,\be^{n-1}$.  The constant $c$ will be chosen later.  We define $U$ to be the intersection of the cylinder with base $B_{ch^{2/3}} ^{n-1} (\tilde\bp)$ and height $h$ with $X$, 
$$U := \left( B_{ch^{2/3}} ^{n-1} (\tilde\bp) \times I_h \right) \cap X,$$
where $I_h$ is the interval of length $h$. 
Let 
$$V := X \setminus U,$$
and let 
$$\beta := \max \left\{ \int_V \phi_1 ^2, \, \, \int_V \phi_2 ^2 \right\}.$$

\textbf{Claim:  }  There exists a constant $A$ which depends only on $n$ and $Y$ such that if 
$$c>A \textrm{ and } h \leq \left( \frac{1}{2c} \right) ^{3/2},$$
then 
$$\beta < \frac{1}{10}.$$

\textbf{Proof of Claim:  }  
We shall begin by assuming 
$$h \leq \left( \frac{1}{2c} \right) ^{3/2},$$ 
which guarantees 
$$1 - ch^{2/3} \geq \frac{1}{2}.$$
By definition of the simplex as the convex hull of its defining points, since $\bp^0, \ldots, \bp^{n-1}$ are contained in the span of $\be^1, \ldots , \be^{n-1}$, the diameter of the simplex is $1$, and $h \leq 0.1$, it follows that 
\begin{equation} \label{eq:hc} h(\bx) \leq h (1- ch^{2/3}), \quad \forall \bx \in V. \end{equation} 
By the one dimensional Poincar\'e inequality and since $\int_U \phi_i ^2 = 1 - \int_V \phi_i ^2$, 
\begin{equation} \label{eq:pi1} \lambda_i \geq  \frac{\pi^2}{h^2} \left( 1 - \int_V \phi_i ^2 \right) + \frac{\pi^2}{h^2 (1-ch^{2/3})^2} \int_V \phi_i ^2, \quad i=1,2. \end{equation} 
On the other hand, $X$ contains a cylinder  
$$\Sigma \cong [0, h(1-h^{2/3})] \times h^{2/3} Y,$$
where $h^{2/3} Y$ is the base scaled by $h^{2/3}$.  One computes explicitly
\begin{equation} \label{eq:pi2} \lambda_2 (\Sigma) = \frac{\pi^2}{h^2(1-h^{2/3})^2} + \frac{C_2}{h^{4/3}}, \end{equation}  
where $C_2$ is the second Dirichlet eigenvalue of $Y$.  Consequently, (\ref{eq:pi1}) and (\ref{eq:pi2}) imply that for $i=1,2$, 
$$\frac{\pi^2}{h^2} \left( 1- \int_V \phi_i ^2 \right)+ \frac{\pi^2}{h^2 (1-ch^{2/3})^2} \int_V \phi_i ^2 \leq \lambda_i \leq \lambda_2 (\Sigma) = \frac{\pi^2}{h^2(1-h^{2/3})^2} + \frac{C_2}{h^{4/3}},$$
which shows that 
$$\left( \frac{\pi^2}{(1-ch^{2/3})^2} - \pi^2 \right) \int_V \phi_i ^2 \leq \frac{\pi^2}{(1-h^{2/3})^2} + C_2 h^{2/3} - \pi^2 \leq \left(C_2 + 3 \pi^2\right) h^{2/3},$$
where the final inequality follows since $h < \frac{1}{10}$.  On the other hand, 
$$h \leq \left( \frac{1}{2c} \right)^{3/2},$$
which shows that 
$$ch^{2/3} \leq \frac{1}{2} \textrm{ and } 2 \pi^2 c h^{2/3} \leq \frac{\pi^2}{(1-ch^{2/3})^2} - \pi^2.$$
We have for $i=1,2$, 
\begin{equation} \label{eq:beta-est}\int_V \phi_i ^2 \leq \frac{C_2 + 3 \pi^2}{2c \pi^2} \implies \beta \leq \frac{C_2 + 3 \pi^2}{2c \pi^2}. \end{equation} 
Therefore, letting 
$$A = \frac{5 (C_2 + 3 \pi^2)}{\pi^2}, \textrm{ then } c>A \textrm{ and } h \leq \left( \frac{1}{2c} \right) ^{3/2} \implies \beta < \frac{1}{10}.$$
\qed

Consider the so-called ``drift Laplacian'' $\Delta_U$ on $U$ with respect to the weight function $f = - 2 \log \phi_1$,  
$$\Delta_U := \Delta + 2 \n \log \phi_1 \n .$$ 
 Let $\mu$ be the first non-zero Neumann eigenvalue of $\Delta_U$ on $U$, and let 
$$\psi := \frac{\phi_2}{\phi_1}.$$
Then, $\psi$ satisfies 
$$\Delta \psi + 2 \n \log \phi_1 \n \psi = -(\lambda_2 - \lambda_1) \psi.$$
Let 
$$\alpha := \frac{\int_U \psi \phi_1 ^2}{\int_U \phi_1 ^2}, \qquad \tilde{\psi} := \psi - \alpha.$$
Then
\begin{equation} \label{eq:ps1} \int_U \tilde{\psi} \phi_1 ^2 = 0. \end{equation} 
Thus, by the weighted variational principle, since $\tilde{\psi}$ satisfies (\ref{eq:ps1}), 
\begin{equation}\label{eq:th3-0} 
\mu \leq \frac{\int_U |\n \tilde{\psi}|^2 \phi_1 ^2}{\int_U \tilde{\psi}^2 \phi_1 ^2}
=\frac{\int_U |\n {\psi}|^2 \phi_1 ^2}{\int_U \tilde{\psi}^2 \phi_1 ^2}.\end{equation}

We have
\begin{equation} \label{eq:th3-1} \int_U |\n \psi|^2 \phi_1 ^2 \leq \int_X |\n \psi|^2 \phi_1 ^2 = (\lambda_2 - \lambda_1) \int_X \psi^2 \phi_1 ^2 = \lambda_2 - \lambda_1. \end{equation}
Using the claim we have, 
$$\int_U \psi^2 \phi_1 ^2 = \int_X \psi^2 \phi_1 ^2 - \int_V \psi^2 \phi_1^2 = 1 - \int_V \phi_2 ^2 > \frac{9}{10}.$$
Since $\int_X \phi_1 \phi_2 = 0$, 
$$\left| \int_U \phi_1 \phi_2 \right| = \left| \int_V \phi_1 \phi_2 \right|,$$
so by the Cauchy inequality, 
$$\alpha = \frac{\int_U \phi_1 \phi_2 }{\int_U \phi_1 ^2} \leq \frac{ \sqrt{\int_V \phi_1^2 } \sqrt{\int_V \phi_2^2} }{9/10} \leq \frac{1}{9}.$$
Thus,
\begin{equation} \label{eq:th3-2}\int_U \tilde{\psi}^2 \phi_1 ^2 \geq \frac{9}{10} - \frac{1}{9} > \frac{1}{2}. \end{equation} 
Putting together (\ref{eq:th3-0}), (\ref{eq:th3-1}), and (\ref{eq:th3-2}), we have
\begin{equation} \label{eq:muleq} \mu \leq 2 (\lambda_2 - \lambda_1).  \end{equation} 
Since $\mu = \lambda_2 (U) - \lambda_1(U)$, and $U$ is convex, by \cite{ac}
\begin{equation} \label{eq:mugeq} \mu \geq \frac{3\pi^2}{d(U)^2}. \end{equation} 
We estimate that   
$$d(U) ^2 \leq  (2c h^{2/3})^2 + h^2.$$
This estimate for $d(U)$ together with (\ref{eq:muleq}) and (\ref{eq:mugeq}) give 
\begin{equation} \label{eq:gap0} \lambda_2 - \lambda_1 \geq \frac{3 \pi^2}{ 2( 2 c h^{2/3})^2 + h^2}. \end{equation} 
Fixing $c$, (\ref{eq:gap0}) demonstrates that $\xi(X) \geq C h^{-4/3} \to \infty$ as $h \to 0$, for a constant $C$ which depends only on $n$ and the base of the simplex $Y$.  
\qed 

\section{Theorem \ref{th:emin} is true for \em short triangles \em}  
Refining estimates from the proof of Theorem \ref{th:simplex}, we demonstrate that if a triangle is sufficiently ``short,'' its fundamental gap is strictly larger than ${64 \pi^2}/{9}$.  

\begin{proposition} \label{pr:short}
Let $T$ be a triangle with vertices $(0,0)$, $(1,0)$, and $(x_0, h)$, where $h\leq 0.005$, and $0.5 \leq x_0 \leq 1$.   Let $\lambda_1$ and $\lambda_2$ be the first two Dirichlet eigenvalues of $T$.  
Then, 
$$\lambda_2 - \lambda_1 > \frac{64 \pi^2}{9}.$$
\end{proposition} 

\begin{proof} 
Define
$$U := \{ (x,y) \in T : x_0-ch^{2/3} \leq x \leq x_0+ch^{2/3} \}, \quad V := T - U,$$
where the constant $c$ will be specified later.  \em The main idea, as in the proof of Theorem 2, is that $\lambda_2 - \lambda_1$  is well approximated by the first positive Neumann eigenvalue of $U$.  \em   Assume the eigenfunctions $\phi_i$ for $\lambda_i$ satisfy 
$$\int_T \phi_i ^2 = 1, \quad i=1,2,$$
and let
$$\beta := \max \left\{ \int_V \phi_i ^2 \right\}_{i=1,2}, \quad \alpha := \frac{\int_U \phi_2 \phi_1 }{ \int_U \phi_1 ^2}, \quad f:= \frac{\phi_2}{\phi_1}.$$
Noting that 
$$\int_U (f-\alpha) \phi_1 ^2 = 0,$$ 
the weighted variational principle for the first positive Neumann eigenvalue $\mu(U)$ of the drift Laplacian $\Delta_U$ with weight function $-2 \log \phi_1$ gives 
$$\mu (U) \leq \frac{ \int_U \ |\n f| \phi_1 ^2}{\int_U (f - \alpha)^2 \phi_1 ^2} \leq \frac{ \int_T |\n f |^2 \phi_1 ^2}{\int_U (f-\alpha)^2 \phi_1^2}= \frac{ \lambda_2 - \lambda_1}{\int_U (f-\alpha)^2 \phi_1 ^2 }.$$
We compute the denominator
$$\int_U (f-\alpha)^2 \phi_1 ^2=  \int_U \phi_2 ^2 - \alpha^2 \left( \int_U \phi_1 ^2 \right) \geq 1-\beta -  \alpha^2,$$
where we have used the definition of $\beta$ with $\int_U \phi_2^2 = 1- \int_V \phi_2 ^2$, together with $\int_U \phi_1 ^2 \leq \int_T \phi_1^2 = 1$.  So, we have 
$$\mu(U) \leq \frac{ \lambda_2 - \lambda_1}{ 1 - \beta - \alpha^2}.$$
By Corollary 1 of \cite{lr} and Corollary 1.4 of \cite{ac}, 
$$\mu(U) \geq \frac{3 \pi^2}{d^2 (U)},$$
which implies
$$\lambda_2 - \lambda_1 \geq (1 - \beta - \alpha^2) \frac{3\pi^2}{d^2 (U)}.$$
Since $\phi_1$ and $\phi_2$ are $\cL^2$ orthogonal,  
$$|\alpha| = \frac{ | \int_V \phi_1 \phi_2|}{\int_U \phi_1 ^2},$$
which by the Cauchy Schwarz inequality and definition of $\beta$ gives 
$$|\alpha| \leq \frac{ \beta}{1-\beta}.$$
Consequently,
\begin{equation}\label{eq:htgap} \lambda_2 - \lambda_1 \geq \left( 1 - \beta - \frac{\beta^2}{(1-\beta)^2} \right) \frac{3 \pi^2}{d^2 (U)}. \end{equation} 
\em Proceeding by contradiction, we assume \em 
\begin{equation} \label{eq:contra} \lambda_2 - \lambda_1 \leq \frac{64 \pi^2}{9} \implies \lambda_2 \leq \lambda_1 + \frac{64 \pi^2}{9}. \end{equation} 
By trigonometry, $T$ contains a rectangle 
$$R \cong [0, h^{2/3}] \times [0, h - h^{5/3}].$$
By domain monotonicity, 
$$\lambda_1 \leq \lambda_1 (R) =  \frac{\pi^2}{h^2 (1-h^{\frac{2}{3}})^2} + \frac{\pi^2}{h^{\frac{4}{3}}}.$$
The height of $V$ is at most 
$$h\left( 1 - c h^{2/3} \right),$$
since $x_0 \in [0.5, 1]$.  By the one dimensional Poincar\'e inequality for $i=1,2$, 
$$\frac{\pi^2}{h^2} \int_U \phi_i ^2 + \frac{\pi^2}{h^2(1 - ch^{\frac{2}{3}})^2} \int_V \phi_i ^2 \leq \lambda_i.$$
Since 
$$\lambda_1 < \lambda_2 \leq \lambda_1 + \frac{64 \pi^2}{9} \leq \lambda_1 (R) + \frac{64 \pi^2}{9},$$ 
by definition of $\beta$, since for either $i=1$ or $i=2$, $ \int_U \phi_i ^2 = 1- \beta$, 
$$  \frac{\pi^2}{h^2} (1-\beta) + \frac{\pi^2}{h^2(1-ch^{\frac{2}{3}})^2} \beta \leq \lambda_2 \leq \frac{\pi^2}{h^2 (1-h^{\frac{2}{3}})^2} + \frac{\pi^2}{h^{\frac{4}{3}}} + \frac{64 \pi^2}{9},$$
so we have
$$\frac{-c^2 h^{4/3} + 2ch^{2/3}}{(1-ch^{2/3})^2} \beta \leq \frac{1}{(1-h^{2/3})^2}  + h^{2/3} + \frac{64 h^2}{9} - 1.$$
Since $h < 0.1$, we have
$$\beta \leq \left( \frac{(1-ch^{2/3})^2}{-c^2 h^{4/3} + 2ch^{2/3}} \right) \left(4h^{2/3} + \frac{64 h^2}{9} \right)$$
which simplifies to 
\begin{equation} \label{beta-1} \beta \leq \left( \frac{ (1-ch^{2/3})^2}{c(2-ch^{2/3})} \right) \left( 4 + \frac{64 h^{4/3}}{9} \right). \end{equation} 
Since we assume $h \leq 0.005$, then for any $c < 34$, $ch^{2/3} <1$.   In particular, fixing $c=10$, we compute that (\ref{beta-1}) gives for any $h \leq 0.005$, 
$$\beta < 0.12.$$
Then, 
$$\left( 1 - \beta - \frac{\beta^2}{(1-\beta)^2} \right) > 0.86,$$
so (\ref{eq:htgap}) becomes 
$$\lambda_2 - \lambda_1 \geq \frac{3*0.86* \pi^2}{d^2(U)}.$$
Since 
$$d^2(U) \leq h^2 + (2ch^{2/3})^2,$$
we compute that for any $h \leq 0.005$, 
$$\lambda_2 - \lambda_1 > \frac{64 \pi^2}{9},$$
which contradicts (\ref{eq:contra}).  Thus if $h \leq 0.005$, then we have $\xi(T) > \frac{64 \pi^2}{9}$.  
\end{proof} 

\section{The equilateral triangle is a strict local gap minimizer}
The main result of this section demonstrates that the equilateral triangle is a strict local minimum for the gap function on the moduli space of triangles.  In the proof, we consider all possible linear deformations of the equilateral triangle and demonstrate that in any direction, such a deformation strictly increases the gap function.  

\subsection{Linear deformation theory}  
\begin{figure}[ht]
\centering
\includegraphics[width=7cm]{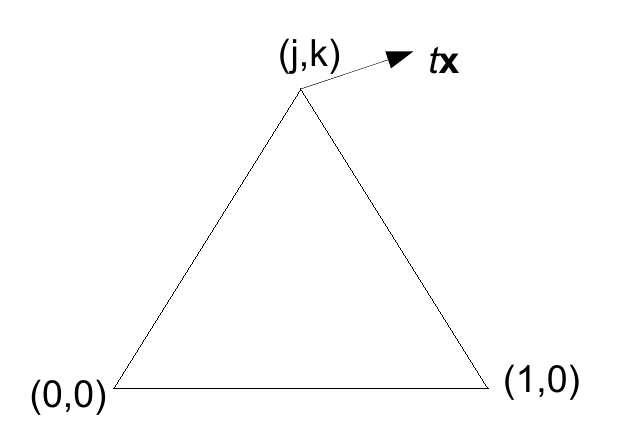}\caption{Linear deformation of a triangle.}\label{tridef}
\end{figure}
Let $T$ be a triangle with vertices $(0, 0),$ $(1,0),$ and $\bz = (j, k),$ and side lengths $A \leq B \leq 1.$    Consider a deformation to the triangle $T(t)$ which has vertices $(0,0),$ $(1,0),$ and $\bz + t \bx,$ where
$$\bx = (a,b), \qquad a^2 + b^2 = 1, \quad t \geq 0.$$
The direction of the deformation is given by $(a, b)$, while the magnitude is given by $t \geq 0$.  The linear transformation which maps the triangle $T$ to the triangle $T(t)$ is represented by the matrix 
$$A=
\begin{bmatrix}
1&\frac{ta}{k}\\
0&1+\frac{tb}{k}
\end{bmatrix}.$$
We may view the linear transformation $T \mapsto T(t)$ as a change of the (Euclidean) Riemannian metric on $\R^2$.  In other words, $T(t)$ is isomorphic to $T$ with the metric, 
$$g = (dx)^2+2\frac{ta}{k} dxdy+\left(\left(1+\frac {tb}{k}\right)^2+\frac{t^2a^2}{k^2}\right)(dy)^2$$
We compute 
$$\det g=\left(1+\frac {tb}{k}\right)^2.$$
Thus 
$$g^{-1}  = \begin{bmatrix} \frac{k^2 + 2bkt + t^2}{(k+bt)^2} & \frac{-tak}{(k+tb)^2} \\  \frac{-tak}{(k+tb)^2} & \frac{k^2}{(k+tb)^2} \end{bmatrix} = \begin{bmatrix} A & B \\ B & D \end{bmatrix}.$$
If the eigenvalues of the original triangle and the deformation triangle are respectively $\lambda_i$ and $\lambda_i (t)$, then they satisfy

\begin{equation} \label{eq:gen-ev-var}
\gamma_- \lambda_i \leq \lambda_i (t) \leq \gamma_+ \lambda_i, 
\end{equation}
where $\gamma_{\pm}$ are the eigenvalues of $g^{-1}$.  It follows that 
$$\left| \lambda_i (t) - \lambda_i \right| \leq (\gamma_+ - \gamma_-) \lambda_i.$$
We compute 
$$\gamma_{\pm} = \frac{A+D}{2} \pm \frac{\sqrt{ (A-D)^2 + 4B^2}}{2}.$$
Substituting the values of $A$, $B$, and $D$ gives 
$$\gamma_+ - \gamma_-= t \frac{\sqrt{4k^2 + t^2 + 4bkt}}{(k+tb)^2},$$
The relationship between integration over $T(t)$ and $T$ differs by a linear factor, 
$$\int_{T(t)} = \left( 1 + \frac{tb}{k} \right) \int_T,$$
where throughout this paper, integration is with assumed to be respect to the standard Lebesgue measure $dx dy$ on $\R^2$.  The Laplace-Beltrami operator associated to a Riemannian metric (in dimension $n$) is 
$$\Delta = \frac{1}{\sqrt{\det(g)}} \sum_{i,j=1} ^n \pa_i g^{ij} \sqrt{\det(g)} \pa_j,$$
so one computes the Laplacian for the deformation metric $g$ is 
$$\Delta = \frac{1}{(1+tb/k)^2} \left( \left( \left( 1 + \frac{tb}{k} \right)^2 + \frac{t^2 a^2}{k^2} \right) \pa_x ^2 - \frac{2ta}{k} \pa_x \pa_y + \pa_y ^2 \right).$$

Henceforth we shall use $\Delta_0 = \pa_x ^2 + \pa_y^2$ for the Euclidean Laplacian, $\Delta$ for the Laplacian associated to a deformation metric, and $\nabla f = (f_x, f_y)$ the gradient (with respect to the standard Euclidean metric).  In \S 6, we shall use this linear deformation theory to complete the proof of Theorem \ref{th:emin}.  Presently, we specialize to the equilateral triangle which we call $T$, and whose vertices are
$$(0,0), \quad (1, 0), \quad \textrm{and} \quad \left( \frac{1}{2}, \frac{\sqrt{3}}{2} \right).$$
A triangle obtained by a linear deformation of $T$, with vertices 
$$(0,0), \quad (1, 0), \quad \textrm{and} \quad \left( \frac{1}{2}, \frac{\sqrt{3}}{2} \right) + t(a,b),$$
is equivalent to $T$ with the metric 
$$
g=(dx)^2+\frac{4ta}{\sqrt{3}} dxdy+\left(\left(1+\frac {2tb}{\sqrt{3}}\right)^2+\frac{4t^2a^2}{3}\right)(dy)^2.$$
We compute, 
$$\det g=\left(1+\frac {2tb}{\sqrt{3}}\right)^2,$$
and 
\begin{equation} \label{eq:alphai} \alpha_i := \frac{ \lambda_i (t) - \lambda_i }{t}, \quad |\alpha_i| \leq \frac{4 \sqrt{3 + 2 \sqrt{3} b t + t^2}}{(\sqrt{3} + 2tb)^2}.  \end{equation}  
The associated Laplace operator 
$$\Delta = \frac{1}{(1+2bt/\sqrt{3})^2} \left( \left( \left( 1 + \frac{2tb}{\sqrt{3}} \right)^2 + \frac{4t^2 a^2}{3} \right) \pa_x ^2 - \frac{4ta}{\sqrt{3}} \pa_x \pa_y + \pa_y ^2 \right).$$
Let
$$L := \frac{1}{\left( 1 + \frac{2tb}{\sqrt{3}} \right)^2} \left( \frac{4ta^2}{3} \pa_x ^2 - \frac{4a}{\sqrt{3}} \pa_x \pa_y - \left( \frac{4b}{\sqrt{3}} + \frac{4tb^2}{3} \right) \pa_y ^2 \right),$$
then
\begin{equation} \label{eq:lapl-var1} \Delta = \pa_x ^2 + \pa_y ^2 + tL = \Delta_0 + tL_1 + t^2 L_2, \end{equation} 
where 
\begin{equation} \label{eq:l1} L_1 = \frac{4\sqrt{3}}{(\sqrt{3} + 2tb)^2} \left( -a \pa_x \pa_y - b \pa_y ^2 \right), \quad L_1(t=0) = \frac{4 \sqrt{3}}{3} \left( -a \pa_x \pa_y - b \pa_y^2 \right).\end{equation}
and
\begin{equation} \label{eq:l2} L_2 = \frac{4}{(\sqrt{3} + 2tb)^2} \left( a^2 \pa_x ^2 - b^2 \pa_y ^2 \right). \end{equation} 
By the variational principle since $\lambda_1$ is smooth, the first eigenvalue for the $T(t)$, which we write as $\lambda_1 (t)$ satisfies
\begin{equation} \label{eq:l1-var} \lambda_1 (t) = \lambda_1 + t \dot{\lambda}_1 + t^2 O_1, \quad \dot{\lambda}_1 = - \int_T \phi_1 \left( L_1|_{t=0} \right) \phi_1, \end{equation} 
where $\phi_1$ is an eigenfunction for $\lambda_1$ with unit $\cL^2$ norm.  If $\lambda_2$ is simple, then we also have
$$\lambda_2 (t) = \lambda_1 + t \dot{\lambda}_2 + t^2 O_2, \quad \dot{\lambda}_2 = - \int_T \phi_2 \left( L_1|_{t=0} \right) \phi_2,$$
where $\phi_2$ is an eigenfunction for $\lambda_2$ with unit $\cL^2$ norm.  

In general, $\lambda_2$ is not differentiable because the second eigenspace may have dimension greater than one; this is the case for the equilateral triangle.  Nonetheless, we may use the variational principle to show that the equilateral triangle is a strict local minimum for the gap function restricted to the moduli space of triangles.  

\begin{theorem} \label{th:lmin} The equilateral triangle is a strict local minimum for the gap function on the moduli space of triangles.
\end{theorem}   

We will prove the theorem by applying the following proposition together with explicit calculations for the eigenvalues and eigenfunctions of the equilateral triangle.  

\begin{proposition} \label{prop:lmin} For any deformation of the equilateral triangle which preserves diameter, for the corresponding $L_1$, 
$$ - \left( \int_T \phi_2 (L_1|_{t=0}) \phi_2  - \int_T \phi_1 (L_1|_{t=0})  \phi_1 \right) > 0, $$
for all eigenfunctions $\phi_i$ for $\lambda_i (T)$, $i=1,2$, with unit $\cL^2$ norm.
\end{proposition}

\textbf{Proof that Proposition \ref{prop:lmin} implies Theorem \ref{th:lmin}:}  
Let $\phi_1$ and $\phi_2$ be eigenfunctions for the first two Dirichlet eigenvalues $\lambda_1$ and $\lambda_2$, respectively, for the equilateral triangle $T$.  Assume the eigenfunctions have unit $\cL^2$ norm on $T$.   Let $f_1$ and $f_2$ be eigenfunctions for the first two Dirichlet eigenvalues of $T(t)$.  Let 
$$\epsilon = \epsilon(t) := \frac{- \int \phi_1 f_2  }{\int \phi_1 f_1},$$
where above and indeed throughout this paper, integration is over the equilateral triangle $T$ \em unless otherwise indicated.  \em   For simplicity, in this section we shall use $L_1$ to denote $L_1|_{t=0}$.  Since $A$ is a linear transformation from $T$ to $T(t),$ $f_1$ and $f_2$ are orthogonal with respect to $dx dy$, so by the convergence of $f_1 \to \phi_1$, 
\begin{equation} \label{eq:leeps} \epsilon = \frac{\int f_2  (f_1  - \phi_1)}{\int \phi_1 f_1 } = O(t) \textrm{ as } t \to 0.\end{equation}
Note that 
$$\int (f_2 + \epsilon f_1 )\phi_1 = 0,$$
so by the variational principle, 
\begin{equation} \label{eq:lel2} \lambda_2 \leq \frac{- \int (f_2  + \epsilon f_1 ) \Delta_0 (f_2 + \epsilon f_1) }{\int |f_2  + \epsilon f_1| ^2}. \end{equation} 
Since these functions are uniformly bounded in $\cC^k$ for any fixed $k$, the Laplace operator on $T(t)$ 
$$\Delta = \Delta_0 + tL_1 + O(t^2).$$
Then, 
$$-\int (f_2 + \epsilon f_1) \Delta_0 (f_2 + \epsilon f_1) = -\int (f_2 + \epsilon f_1)(\Delta -  tL_1)(f_2 + \epsilon f_1) + O(t^2)$$
$$= \lambda_2 (t) + t \int f_2 L_1 f_2  + O(t^2).$$
Consequently, 
$$\lambda_2 \leq \lambda_2(t) + t \int  f_2 L_1 f_2 + O(t^2).$$
Since $\lambda_1$ is differentiable with $\dot{\lambda}_1 = -\int \phi_1  L_1 \phi_1$, 
\begin{equation} \label{eq:lel1} \lambda_1 (t) = \lambda_1 + t \dot{\lambda}_1 + O(t^2) \implies \lambda_1 = \lambda_1 (t) + t \int \phi_1 L_1 \phi_1 + O(t^2). \end{equation} 
Then, (\ref{eq:lel2}) and (\ref{eq:lel1}) imply 
\begin{equation} \label{eq:gapmin1} \lambda_2 - \lambda_1 \leq \left[ \lambda_2 (t) - \lambda_1 (t) \right]  + t\left[\int f_2 L_1 f_2 - \int \phi_1 L_1  \phi_1 \right] +  O(t^2).\end{equation}
Since the deformation preserves diameter, we may re-write (\ref{eq:gapmin1}) as 
\begin{equation} \label{eq:genmin} \xi(T) \leq \xi(T(t))  +  t\left[\int f_2 L_1 f_2 - \int \phi_1  L_1  \phi_1 \right] +  O(t^2). \end{equation}  
We can always construct a sequence of eigenfunctions $f_2$ which converge in $\cC^2$ to some eigenfunction $\phi_2$ for $\lambda_2$ with $\int \phi_2 ^2 = 1$.  Consequently, 
$$\xi(T) \leq \xi(T(t)) + t \left[ \int \phi_2 L_1 \phi_2 - \int \phi_1 L_1 \phi_1 \right] + O(t^2).$$
Since for all $\phi_2$, 
$$\int \phi_2 L_1 \phi_2 - \int \phi_1 L_1 \phi_1 < 0,$$
we have $\xi(T) < \xi(T(t))$ for all $t$ sufficiently small.  Finally, we note that we need only consider deformations in directions which preserve the diameter because the gap function is scale invariant.  
We have therefore reduced the theorem to verifying explicit calculations involving the 
eigenfunctions and eigenvalues of the equilateral triangle.  

\subsection{Eigenfunctions and eigenvalues of the equilateral triangle}  
In 1852, Lam\'e computed the eigenfunctions and eigenvalues of the equilateral triangle by real analytically extending the eigenfunctions to the plane using the symmetry of the equilateral triangle \cite{lame}, \cite{lame1}, \cite{lame2}.  The eigenvalues are given by the general formula 
\begin{equation}\label{eq:eqtriev} \lambda = \frac{16 \pi^2}{27} (m^2 + n^2 - mn), \quad m, n \in \Z, \end{equation} such that $m$ and $n$ satisfy
\begin{equation} \label{eq:eqtrimn} m+n \equiv 0 \textrm{ mod } 3, \quad m \neq 2n, \quad n \neq 2m, \quad  m \neq -n. \end{equation} 
The eigenfunctions are 
$$\sum_{(m,n)} \pm \cos \left( \frac{2\pi }{3} \right) \left( nx + \frac{(2n-m)y}{\sqrt{3}} \right) \quad \textrm{and} \quad \sum_{(m,n)} \pm \sin \left( \frac{2\pi }{3} \right) \left( nx + \frac{(2n-m)y}{\sqrt{3}} \right),$$
where the sum alternates over the orbit 
$$(-n, m-n) \mapsto (-n, -m) \mapsto (n-m, -m) \mapsto (n-m, n) \mapsto (m, n) \mapsto (m, m-n),$$
such that $(m,n)$ satisfy (\ref{eq:eqtriev}, \ref{eq:eqtrimn}) for the eigenvalue
$$\lambda = \frac{16 \pi^2}{27} (m^2 + n^2 - mn).$$

\subsubsection{The first eigenspace of the equilateral triangle}
The first Dirichlet eigenvalue of the equilateral triangle with side lengths one is given by (\ref{eq:eqtriev}) with $m=0$ and $n=3$, (or with $m=n=3$), 
\begin{equation} \label{eq:lame} \lambda_1 = \frac{16 \pi^2}{3}. \end{equation}
Since the first Dirichlet eigenvalue is always simple, it follows from Corollary 2 of \cite{pin} that  the first $\cL^2$ normalized eigenfunction of the equilateral triangle $T$ is 
$$\phi_1 (x,y) = \frac{2 \sqrt{2}}{3^{3/4}}\left(  \sin\left( \frac{4\pi y}{\sqrt{3}} \right) - \sin \left( 2 \pi \left( x + \frac{\sqrt{3}y}{3} \right) \right) + \sin \left( 2 \pi \left( x - \frac{\sqrt{3} y}{3} \right) \right) \right).$$
\begin{proposition} 
$$\phi_1 (x,y) =  \frac{2 \sqrt{2}}{3^{3/4}}  \sin \left( \frac{2 \pi \sqrt{3}y}{3} \right) \sin \left( \pi \left(x + \frac{\sqrt{3} y}{3} \right) \right) \sin \left( \pi \left( x - \frac{\sqrt{3}y}{3} \right) \right).$$
\end{proposition}
\begin{proof}
The standard angle addition and subtraction identities for sine and cosine show that, 
$$ \sin \left( \frac{4 \pi \sqrt{3} y}{3} \right) - 2 \cos 2 \pi x \sin \left( 2\pi \frac{\sqrt{3}y}{3} \right)$$
$$= 2 \sin \left( 2 \pi \frac{\sqrt{3}y}{3} \right) \left( \cos \frac{2 \pi \sqrt{3}y}{3} - \cos 2 \pi x \right) $$
$$= \sin \left( \frac{2 \pi \sqrt{3}y}{3} \right) \sin \left( \pi \left(x + \frac{\sqrt{3} y}{3} \right) \right) \sin \left( \pi \left( x - \frac{\sqrt{3}y}{3} \right) \right).$$
The last equality follows from the identity
$$\cos \alpha - \cos \beta = - 2 \sin\left( \frac{1}{2} (\alpha + \beta) \right) \sin \left( \frac{1}{2} (\alpha - \beta) \right).$$
\end{proof}
We compute
$$\int (\phi_1)_y^2 = \int (\phi_1)_x ^2 = \frac{8 \pi^2}{3}, \quad \int (\phi_1)_x (\phi_1)_y = 0.$$ 
$$\int (\phi_1)_{xy}^2 = \frac{32 \pi^4}{9}, \quad \int (\phi_1)_{yy}^2 = \int (\phi_1)_{xx} ^2 = \frac{32 \pi^4}{3}.$$
$$\int \phi_1 = \frac{3^{3/4}}{\sqrt{2} \pi}.$$
$$\int (\phi_1)_{xx} (\phi_1)_{xy} = 0 = \int (\phi_1)_{yy} (\phi_1)_{xy}.$$

\subsubsection{The second eigenspace of the equilateral triangle}
The second Dirichlet eigenvalue of the equilateral triangle is given by (\ref{eq:eqtriev}) with $m=1$ and $n=5$ (or with $m=-1$ and $n=4$), 
 $$\lambda_2 = \frac{112 \pi^2}{9}.$$
This eigenspace has dimension two.  An $\cL^2$ orthonormal basis of eigenfunctions is given by 
$$u(x,y) = \frac{2}{3^{\frac{3}{4}}} \left( \begin{array}{l} \cos \frac{2\pi}{3}(5 x - \sqrt{3} y) - \cos \frac{2\pi}{3}(5 x + \sqrt{3} y)+ \cos \frac{2\pi}{3}(- x +3 \sqrt{3} y) \\
- \cos \frac{2\pi}{3}(- x -3 \sqrt{3} y)+ \cos \frac{2\pi}{3}(- 4x -2 \sqrt{3} y)- \cos \frac{2\pi}{3}(- 4x +2 \sqrt{3} y) \end{array} \right),$$
and
$$v(x,y) = \frac{2}{3^{\frac{3}{4}}} \left( \begin{array}{l} \sin \frac{2\pi}{3}(5 x - \sqrt{3} y) - \sin \frac{2\pi}{3}(5 x + \sqrt{3} y)+ \sin \frac{2\pi}{3}(- x +3 \sqrt{3} y) \\
 - \sin \frac{2\pi}{3}(- x -3 \sqrt{3} y)+ \sin \frac{2\pi}{3}(- 4x -2 \sqrt{3} y)- \sin \frac{2\pi}{3}(- 4x +2 \sqrt{3} y) \end{array} \right).$$

The following calculations will play a key role in the proof of Theorem 2.
$$\int |u_x|^2 = \int |v_y|^2 =  -\frac{6561}{800} + \frac{56 \pi^2}{9}, \quad \int u_x u_y = \frac{-6561 \sqrt{3}}{800}.$$
$$ \int v_x v_y = \int u_y v_y = \frac{6561 \sqrt{3}}{800}, \quad \int u_x v_y = \frac{6561}{800}.$$ 
$$\int u_x v_x = - \frac{6561\sqrt{3}}{800}, \quad \int |u_y|^2 = \int |v_x|^2 = \frac{6561}{800} + \frac{56 \pi^2}{9}.$$
$$\int u_{xy} ^2 = - \frac{5103 \pi^2}{200} + \frac{1568 \pi^4}{81}, \quad  \int u_{yy} ^2 = \frac{5103 \pi^2}{40} + \frac{1568 \pi^4}{27}.$$
$$\int u_{xx} ^2 = \frac{7 \pi^2(-59049 + 44800 \pi^2)}{5400}, \quad  \int v_{xx} ^2 = \frac{7 \pi^2(59049 + 44800 \pi^2)}{5400}.$$
$$\int v_{xy}^2 = \frac{5103 \pi^2}{200} + \frac{1568 \pi^4}{81}, \quad  \int v_{yy} ^2 = -\frac{5103 \pi^2}{40} + \frac{1568 \pi^4}{27}.$$
$$\int v_{xy}u_{xy} = - \frac{5103 \sqrt{3} \pi^2}{200}, \quad \int v_{xx} u_{xx} = - \frac{15309 \sqrt{3} \pi^2}{200}.$$
$$\int v_{yy} u_{yy} = \frac{5103 \sqrt{3} \pi^2}{40}.$$


\subsubsection{The third and higher eigenspaces of the equilateral triangle}
The third (distinct) eigenvalue is given by (\ref{eq:eqtriev}) with $m=n=6$, 
$$\lambda_3 = \frac{64 \pi^2}{3}.$$
The eigenspace has dimension one.   The eigenfunction is
$$A(x,y) = \frac{2}{3^{\frac{3}{4}}} \left( 2 \sin \frac{2\pi}{3}(6 x +4 \sqrt{3} y) -2 \sin \frac{2\pi}{3}(6x +2 \sqrt{3} y) - 2 \sin \frac{2\pi}{3}(2\sqrt{3}\pi y) \right).$$

\subsection{Proof of Proposition \ref{prop:lmin}}
Any real eigenfunction $\varphi$ of $\lambda_2$ is a linear combination, 
$$\varphi = \alpha u + \beta v, \qquad \alpha^2 + \beta^2 = 1,$$
and we may assume without loss of generality that $\alpha \geq 0.$  We compute 
$$\int \varphi_y^2 =  \frac{56 \pi^2}{9} + \left[\alpha^2 - \beta^2 + 2 \sqrt{3} \alpha \beta \right] \frac{6561}{800},$$
and
$$\int \varphi_x \varphi_y = \left[ \beta^2 - \alpha^2 + \frac{2 \alpha \beta}{\sqrt{3}} \right] \frac{6561 \sqrt{3}}{800}.$$
As previously observed, we need only consider those deformations in directions which preserve diameter, and by symmetry, we need only consider those directions $\theta$ with $\cos \theta \geq 0$.  These are deformations in directions $\theta \in [-\pi/2, -\pi/6]$, so the direction vector $(a,b)$ satisfies $a^2 + b^2 = 1$ , $a \geq 0$, and $a + \sqrt{3}b \leq 0$.  We compute the minimum of 
$$I := - \int \varphi ( L_1 |_{t=0})  \varphi + \int \phi_1 ( L_1|_{t=0})   \phi_1 =   \left( \frac{2 b \lambda_1 }{\sqrt{3}}  - \frac{4 b }{\sqrt{3}} \int |\varphi_y|^2 - \frac{4 a }{\sqrt{3}} \int   \varphi_x \varphi_y \right) = $$
$$ -\left( \frac{25600\pi^2 + (\alpha^2 - \beta^2 + 2\sqrt{3} \alpha \beta)59049}{1800\sqrt{3}} \right) b - \frac{6561}{200} \left( \beta^2 - \alpha^2 + \frac{2\alpha\beta}{\sqrt{3}}\right) a.  $$
Since $(a, b)$ is a unit vector, $a \geq 0$, and $b \leq -\frac{a}{\sqrt{3}}$, it follows that $b = -\sqrt{1 - a^2}.$  So, we determine the minimum of 
$$I = \frac{ \left(25600\pi^2 + (\alpha^2 - \beta^2 + 2 \sqrt{3} \alpha \beta)59049 \right) \sqrt{1-a^2} + 59049 \sqrt{3} (\alpha^2 - \beta^2 - \frac{2\alpha\beta}{\sqrt{3}} )a}{1800\sqrt{3}},$$
subject to the constraints
$$\alpha^2 + \beta^2 = 1, \quad  0 \leq a \leq \frac{\sqrt{3}}{2}.$$ 
Introducing the polar coordinates,
$$\cos(t) := \alpha, \quad \sin(t) := \beta,$$
we compute that $I$ is minimized for 
$$t = \frac{\pi}{2}, \quad a = \frac{\sqrt{3}}{2},$$
and the minimum is
$$\frac{25600 \pi^2 - 236 196}{3600 \sqrt{3}} > 2.64>0.$$
\qed 

\begin{remark}  In the proof of the theorem, we have shown that for any triangle $T \in \mathfrak{M}$ with vertices $(0,0)$, $(1,0)$, and $(x,y)$ where $(x - 1/2)^2 + (y-\sqrt{3}/2)^2 \leq t^2$, 
$$\lambda_2(T) - \lambda_1(T) \geq \frac{64 \pi^2}{9} + t (2.64) + O(t^2).$$ 
In the arguments below, we use the calculations in \S 2 to precisely estimate the error $O(t^2)$.  
\end{remark}  
\section{Theorem \ref{th:emin} is true for \em almost equilateral triangles \em}
The following proposition is the last step we need to reduce the proof of Theorem 2 to finitely many numerical calculations.
\begin{proposition} \label{pr:nequi}
Let $T$ be any triangle with vertices $(0, 0)$, $(1, 0)$ and $(x,y)$ such that 
$$x^2 + y^2 \leq 1, \quad (x-1)^2 +y^2 \leq 1,$$
and 
$$t:= \sqrt{ \left(x-\frac{1}{2} \right)^2 + \left(y - \frac{\sqrt{3}}{2} \right)^2} \leq 0.0004.$$
Then, the gap function $\xi(T)$ satisfies
$$\xi(T) > \frac{64 \pi^2}{9}.$$
\end{proposition}

\textbf{Proof:  }  We shall again use $\lambda_i$, $i=1$, $2$ for the eigenvalues of the equilateral triangle and $\lambda_i (t)$, $i=1$, $2$, for the corresponding eigenvalues of a triangle $T(t)$ which satisfies the hypotheses of the proposition.  

Since $\lambda_1$ is differentiable, we have 
$$\lambda_1 (t) = \lambda_1 + t \dot{\lambda}_1 + O_1 (t^2).$$
By the variational principle, 
$$\lambda_1 (t) \leq \frac{- \int_{T(t)} \phi_1 \Delta \phi_1}{\int_{T(t)} \phi_1 ^2},$$
which simplifies to 
$$\lambda_1 (t) \leq  - \int \phi_1 \Delta \phi_1,$$
since integration over $T$ and $T(t)$ differ by linear factors which cancel in the numerator and denominator, and $\int \phi_1 ^2 = 1$.  We then have
$$\lambda_1 (t) \leq -\int \phi_1 (\Delta_0 + tL) \phi_1,$$
where $L = L_1 + L_2$ is defined in (\ref{eq:l1}, \ref{eq:l2}).  So, we compute directly 
$$\lambda_1 (t) \leq \lambda_1 - t \int \phi_1 \left( L_1 |_{t=0} \right) \phi_1 - t \int \phi_1 \left( L - L_1|_{t=0} \right) \phi_1.$$
Thus, we have made explicit 
$$|O_1 (t^2)| \leq  \left| t \int \phi_1 \left( L - L_1|_{t=0} \right) \phi_1 \right|.$$
We have
$$L - L_1|_{t=0} = \frac{4ta^2 \pa_x ^2 - 4tb^2 \pa_y^2}{(\sqrt{3} + 2tb)^2} + \left( \frac{-4\sqrt{3}}{(\sqrt{3} + 2tb)^2} + \frac{4}{\sqrt{3}} \right) \left(a \pa_x \pa_y + b \pa_y ^2 \right).$$

We estimate using the calculations for the first eigenfunction of the equilateral triangle and $a^2 + b^2 =1$, 
$$\left| O_1(t^2) \right| \leq t \left( \left| \frac{4 \sqrt{3}}{(\sqrt{3} + 2tb)^2} - \frac{4}{\sqrt{3}} \right| |b| \int \left( \phi_1 \right)_y ^2 \right) + t^2  \frac{4(a^2 - b^2)}{(\sqrt{3} + 2tb)^2} \int \left( \phi_1 \right)_y ^2 $$
$$\leq t^2 \left( \frac{16 \sqrt{3} + 16t }{\sqrt{3} (\sqrt{3} + 2tb)^2} \frac{8 \pi^2}{3} + \frac{4}{( \sqrt{3} + 2tb)^2} \frac{8 \pi^2}{3} \right),$$
which for $t \leq 0.0004$ gives 
\begin{equation}\label{eq:o1t} |O_1 (t^2)|  \leq t^2 (175.95). \end{equation} 

\subsection{Estimates for the second eigenspace}
Since $\lambda_2$ of the equilateral triangle is not differentiable, the estimates for its error term require a bit more work.  The main idea is to expand the first two eigenfunctions for the linearly-deformed  triangle using the orthonormal basis of eigenfunctions for the equilateral triangle.  We then use the Poincar\'e inequality and our explicit calculations for the eigenfunctions of the equilateral triangle to estimate the error.  

\subsubsection{The first eigenfunction of the linearly deformed triangle}
Our eventual goal is to estimate $\lambda_2 (t)$ from below.   To accomplish this, we require not only estimates for the second eigenspace of the linearly deformed triangle, $T(t)$, but also estimates for its first eigenfunction.  Let $f$ be the first eigenfunction of $T(t)$ and write 
$$f = \phi_1 + tg, \quad \alpha_i := \frac{\lambda_i (t) - \lambda_i}{t} \textrm{ for } i=1,2,$$
with 
\begin{equation} \label{eq:gortho} \int \phi_1 ^2 = 1, \quad \int \phi_1 g = 0.  \end{equation} 
As usual, integration is over the equilateral triangle $T$ with respect to the standard measure $dx dy$, and we use $|| \cdot ||$ to denote the $\cL^2$ norm over $T$.  Since we assume $t \leq 0.0004$, by (\ref{eq:alphai}) 
\begin{equation} \label{eq:alpha2}  | \lambda_i (t) - \lambda_i | = t |\alpha_i| \leq t (2.32) \lambda_i. \end{equation} 

We compute 
\begin{equation} \label{eq:flf1} \int f L f = \int \phi_1 L \phi_1 + 2t \int g L \phi_1 + t^2 \int g L g.  \end{equation}
Since $\int \phi_1 g = 0$, by the variational principle (\ref{varprin}) 
$$\lambda_2 \leq \frac{ \int |\n g|^2}{\int g^2},$$
which gives the Poincar\'e inequality for $g$,
\begin{equation} \label{eq:gest0} || g || \leq \frac{1}{\sqrt{\lambda_2}} || \n g ||. \end{equation}
To  estimate $|| \n g||$, we use the definition of $f$ and $g$ to compute
\begin{equation} \label{eq:1estg} (\Delta_0 + \lambda_1)g = - \alpha_1 f - L\phi_1 - t L g. \end{equation}
By definition of $g$ and (\ref{eq:gortho}), 
\begin{equation} \label{eq:2estg} \int fg = \int (\phi_1 + tg)g = t \int g^2. \end{equation} 
We compute using integration by parts and then substituting (\ref{eq:1estg}, \ref{eq:2estg})
$$\int g(\Delta_0 + \lambda_1)g = - \int |\n g|^2 + \lambda_1 \int g^2 = - t \alpha_1 \int g^2 - \int g L \phi_1 - t \int g L g,$$
which gives
\begin{equation} \label{eq:ng} ||\n g||^2 = (\lambda_1 + \alpha_1 t) \int g^2 + \int g L \phi_1 + t \int g L g. \end{equation} 
By definition of $L$, and since $a^2 + b^2 = 1$, for any function $\psi$ which vanishes on $\pa T$, we have
$$\left| \int \psi L \psi \right| \leq \frac{4 \sqrt{3}}{(\sqrt{3} - 2t|b|)^2} ||\psi_x|| ||\psi_y|| + \frac{4|b| \sqrt{3} + 4tb^2}{(\sqrt{3} - 2t|b|)^2} ||\psi_y||^2,$$
and since we always have $||\psi_x|| \leq || \n \psi||$,  $|| \psi_y || \leq || \n \psi||$, and $-1 \leq b \leq 1$, we have 
\begin{equation} \label{eq:lcoeff} \left| \int \psi L \psi \right| \leq \frac{8 \sqrt{3} + 4t}{(\sqrt{3} - 2t)^2} || \n \psi ||^2 \quad \textrm{ for any } \psi \textrm{ which vanishes on } \pa T. \end{equation} 
Applying this to $g$, we have 
$$\left| \int g L g \right| \leq \frac{8 \sqrt{3} + 4t}{(\sqrt{3} - 2t)^2} ||\n g||^2.$$
By the Poincar\'e inequality for $g$ (\ref{eq:gest0}), (\ref{eq:ng}) with the Cauchy inequality, and the above estimate, we have 
$$||\n g||^2 \leq \frac{\lambda_1 + |\alpha_1| t}{\lambda_2} || \n g||^2 + \frac{||\n g||}{\sqrt{\lambda_2}} ||L \phi_1|| + t \frac{8 \sqrt{3} + 4t}{(\sqrt{3} - 2t)^2} ||\n g||^2,$$
which gives 
$$||\n g|| \leq \left( \frac{\lambda_2 - \lambda_1 - |\alpha_1| t}{\lambda_2} - t \left( \frac{8 \sqrt{3} + 4t}{(\sqrt{3} - 2t)^2} \right) \right)^{-1} \frac{|| L \phi_1 ||}{\sqrt{\lambda_2}}.$$
We assume $t \leq 0.0004$, so substituting the estimate (\ref{eq:alpha2}) for $\alpha_1$, we have
$$|| \n g|| \leq \left( \frac{\lambda_2 - \lambda_1 -0.00232 \lambda_1}{\lambda_2} - 0.0046309\right)^{-1} \frac{|| L \phi_1||}{\sqrt{\lambda_2}}.$$ 
Using the values for $\lambda_1$ and $\lambda_2$, we have 
\begin{equation} \label{eq:gest2} || \n g || \leq 0.15948 || L \phi_1||.  \end{equation} 

Next we estimate $|| L \phi_1||$.  For these calculations it is convenient to drop the subscript and write simply $\phi$.  We have 
$$L \phi = \frac{4 t a^2}{(\sqrt{3} + 2tb)^2} \phi_{xx} - \frac{4\sqrt{3}a}{(\sqrt{3} + 2tb)^2} \phi_{xy} - \frac{4\sqrt{3}b + 4tb^2}{(\sqrt{3} + 2tb)^2} \phi_{yy}.$$
Recalling 
$$\int \phi_{xx} \phi_{xy} = \int \phi_{yy} \phi_{xy} = \int \phi_{xx} \phi_{yy} =  0,$$
and $a^2 + b^2 = 1$, we compute 
$$\int (L \phi)^2 \leq \frac{16 t^2 }{(\sqrt{3} - 2t|b|)^4} || \phi_{xx} ||^2 + \frac{48}{(\sqrt{3} - 2t |b|)^4} ||\phi_{xy}||^2 + \frac{48 + 32 \sqrt{3} t + 16t^2}{(\sqrt{3} - 2t|b|)^4} || \phi_{yy}||^2.$$
Assuming $t \leq 0.0004$ and substituting the value of the integrals, we have
$$||L \phi||^2 \leq (1.7861) t^2 \frac{32 \pi^4}{3} + 5.3581 \frac{32 \pi^4}{9} + 5.3581 \frac{32 \pi^4}{3} + 6.1870 t \frac{32 \pi^4}{3} + 1.7861 t^2 \frac{32 \pi^4}{3}.$$
Assuming $t \leq 0.0004$, we have the estimate for $||L \phi||$, 
\begin{equation} \label{eq:Lphi1} || L \phi|| \leq 86.194 \end{equation} 
and this gives us the approximation for $|| \n g||$, 
$$ || \n g || \leq 0.15948 || L \phi || \leq 13.747.$$
Moreover, we have the estimate for $||g||$,
\begin{equation} \label{eq:g1} ||g|| \leq \frac{||\n g||}{\sqrt{\lambda_2}} \leq \frac{(11.907)(3)}{\sqrt{112} \pi} \leq 1.2404. \end{equation} 

\subsubsection{The second eigenspace}  
 Let $F$ be an eigenfunction in the second eigenspace of $T(t)$, and assume $||F||=1$, where as usual, the $\cL^2$ norm is taken over the equilateral triangle $T$. Expanding $F$ in terms of the eigenfunctions of the equilateral triangle, 
 $$F = \varphi + A \phi_1 + t G,$$
 where $\varphi$ is an eigenfunction for $\lambda_2$, and $G$ satisfies
 $$\int G \phi_1 = 0, \quad \int G \phi_2 = 0 \quad \forall \textrm{ eigenfunction $\phi_2$ for $\lambda_2$.}$$
 Then, we have 
 \begin{equation} \label{eq:intGFA} \int GF = t \int G^2 \quad \textrm{and} \quad \int F \phi_1 = A = \int F(f-tg) = - t \int Fg, \end{equation} 
 so by the Cauchy inequality
 $$A^2 \leq t^2 ||g||^2 ||F||^2.$$
By definition of $F$ and $G$, 
$$\int F^2 = 1+ A^2 + t^2 || G || ^2.$$
Combining these, we have 
$$A^2 \leq t^2 || g|| ^2 ||F ||^2 = t^2 ||g ||^2 (1 + A^2 + t^2 ||G||^2),$$
so we obtain the estimate for $A$,
\begin{equation} \label{eq:a-1} A^2 \leq \frac{t^2 ||g||^2 (1 + t^2 ||G||^2) }{1-t^2 ||g||^2} \implies |A| \leq \frac{t ||g|| (1+t ||G||)}{\sqrt{1-t^2 ||g||^2}}. \end{equation} 
Since $G$ is orthogonal to the first two eigenspaces, the variational principle for $\lambda_3$ gives
$$\lambda_3 \leq \frac{\int | \n G|^2}{\int G^2},$$
which implies the Poincar\'e inequality for $G$, 
\begin{equation} \label{eq:pG} ||G||^2 \leq \frac{ ||\n G||^2}{\lambda_3}. \end{equation} 
We estimate $G$ in the same spirit as $g$.   We compute, 
$$(\Delta_0 + \lambda_2) G = (\Delta_0 + \lambda_2) \left( \frac{-A \phi_1 + F}{t} \right),$$
since
$$(\Delta_0 + \lambda_2) \varphi = 0.$$
So, 
$$(\Delta_0 + \lambda_2) G = \frac{1}{t} \left( \lambda_1 A \phi_1 - \lambda_2 A \phi_1 + (\Delta - tL + \lambda_2 (t) - t \alpha_2) F \right),$$
$$=  \frac{(\lambda_1 - \lambda_2)A \phi_1}{t} - (L+ \alpha_2)F.$$
To estimate $||\n G||$ and hence $||G||$ by the Poincar\'e inequality (\ref{eq:pG}), we use the above calculation together with integration by parts (as we did with $g$), 
$$\int G (\Delta_0 + \lambda_2) G =  \int \frac{G (\lambda_1 - \lambda_2)A \phi_1}{t} - \int G (L + \alpha_2) F$$
which since $G$ is orthogonal to $\phi_1$ becomes
$$= - \int G L F - \alpha_2 \int GF.$$
By definition of $F$ and (\ref{eq:intGFA}), this is 
$$ - \int GL(A \phi_1 + \varphi + t G) - t \alpha_2 \int G^2 = - A \int G L \phi_1 - \int G L \varphi - t \int G L G - t \alpha_2 \int G^2.$$
On the other hand, integrating by parts gives
$$\int G(\Delta_0 + \lambda_2) G = - \int |\n G|^2 + \lambda_2 \int G^2.$$
Combining this with the above calculation, we have
$$- \int |\n G|^2 + \lambda_2 \int G^2 = - A \int GL\phi_1 - \int GL \varphi - t \int GLG - t \alpha_2 \int G^2.$$
The estimate for $L$ (\ref{eq:lcoeff}) and the Cauchy inequality imply 
$$||\n G||^2 \leq \lambda_2 ||G||^2 +  ||G ||( |A| || L \phi_1 || + || L \varphi || )+ t \frac{8 \sqrt{3} + 4t}{(\sqrt{3} - 2t)^2}  ||\n G||^2 + t |\alpha_2| ||G||^2.$$
By the Poincar\'e inequality for $G$ (\ref{eq:pG}),
$$||\n G||^2 \leq \frac{\lambda_2}{\lambda_3} ||\n G||^2 +  \frac{||\n G ||}{\sqrt{\lambda_3}}( |A| || L \phi_1 || + || L \varphi || )+ t \frac{8 \sqrt{3} + 4t}{(\sqrt{3} - 2t)^2}  ||\n G||^2 + t \frac{|\alpha_2|}{\lambda_3}||\n G||^2.$$
This gives the estimate 
$$|| \n G || \leq \left( 1 -  \frac{\lambda_2 + t |\alpha_2|}{\lambda_3} - t \left( \frac{8 \sqrt{3} + 4t}{(\sqrt{3} - 2t)^2} \right) \right)^{-1} \left( \frac{ |A| ||L \phi_1 || + || L \varphi||}{\sqrt{\lambda_3}} \right),$$
and
$$|| G|| \leq \frac{1}{\sqrt{\lambda_3}} \left( 1 -  \frac{\lambda_2 + t |\alpha_2|}{\lambda_3} - t \left( \frac{8 \sqrt{3} + 4t}{(\sqrt{3} - 2t)^2} \right) \right)^{-1} \left( \frac{ |A| ||L \phi_1 || + || L \varphi||}{\sqrt{\lambda_3}} \right).$$
Expanding and simplifying we have
$$|| G|| \leq \frac{ (\sqrt{3} - 2t)^2( |A| ||L \phi_1|| + ||L \varphi||) }{(\sqrt{3} - 2t)^2(\lambda_3  - \lambda_2 - t |\alpha_2|) - t \lambda_3 (8 \sqrt{3} + 4t) }.$$
Recalling the estimate (\ref{eq:a-1}) for $A$, 
$$|| G|| \leq  \frac{(\sqrt{3} - 2t)^2}{ (\sqrt{3} - 2t)^2(\lambda_3  - \lambda_2 - t |\alpha_2|) - t \lambda_3 (8 \sqrt{3} + 4t)} \times $$
$$\left(\frac{ t ||L \phi_1|| ||g|| + t^2 ||L \phi_1|| ||g|| ||G||}{\sqrt{1-t^2 ||g||^2}} + ||L \varphi|| \right).$$ 
Collecting the $||G||$ terms,
$$||G || \left( 1 - \frac{ (\sqrt{3} - 2t)^2}{(\sqrt{3} - 2t)^2(\lambda_3  - \lambda_2 - t |\alpha_2|) - t \lambda_3 (8 \sqrt{3} + 4t)} \frac{ t^2 || L \phi_1 || ||g||}{\sqrt{1 - t^2 ||g||^2}} \right)$$
$$ \leq \frac{(\sqrt{3} - 2t)^2}{(\sqrt{3} - 2t)^2(\lambda_3  - \lambda_2 - t |\alpha_2|) - t \lambda_3 (8 \sqrt{3} + 4t) } \left( \frac{ ||L \phi_1 || t ||g || }{\sqrt{1 - t^2 ||g||^2}} + || L \varphi||  \right).$$
This gives the estimate from above for $||G||$ 
\begin{equation} \label{eq:Gestf} \frac{ (\sqrt{3} - 2t)^2  \left( t ||L \phi_1|| ||g|| + \sqrt{ 1- t^2 ||g||^2} ||L \varphi|| \right) }{\left( (\sqrt{3} - 2t)^2(\lambda_3  - \lambda_2 - t |\alpha_2|) - t \lambda_3 (8 \sqrt{3} + 4t) \right) \sqrt{1 - t^2 ||g||^2} - t^2 (\sqrt{3} - 2t)^2 || L \phi_1 || ||g|| }. \end{equation} 
At this point, we may substitute estimates for every term except $||L \varphi||$, which we now estimate.
By definition of $L$, 
$$L \varphi = \frac{4 t a^2}{(\sqrt{3} + 2tb)^2} \varphi_{xx} - \frac{4\sqrt{3}a}{(\sqrt{3} + 2tb)^2} \varphi_{xy} - \frac{4\sqrt{3}b + 4b^2}{(\sqrt{3} + 2tb)^2} \varphi_{yy}.$$
By the triangle inequality for $\cL^2$ and since $t \leq 0.0004$, 
\begin{equation} \label{eq:lvphi} || L \varphi || \leq 0.0013365 || \varphi_{xx} || + 2.3148 || \varphi_{xy}|| + 3.6512  || \varphi_{yy} ||. \end{equation} 
Since $\varphi$ is an $\cL^2$ orthonormal eigenfunction for $\lambda_2$, 
$$\varphi = \alpha u + \beta v, \quad \alpha^2 + \beta^2 = 1.$$
By our calculations for the second eigenspace of the equilateral triangle, 
$$|| \varphi_{xx} || = \sqrt{ \alpha^2 \frac{7 \pi^2(-59049 + 44800 \pi^2)}{5400} + \beta^2  \frac{7 \pi^2 (54049 + 44800 \pi^2)}{5400} - 2 \alpha \beta \frac{15309 \sqrt{3} \pi^2}{200} }.$$
$$|| \varphi_{xy} || = \sqrt{ \alpha^2 \left( -\frac{5103 \pi^2}{200} + \frac{1568 \pi^4}{81} \right) + \beta^2 \left( \frac{5103 \pi^2}{200} + \frac{1568 \pi^4}{81} \right) - 2\alpha \beta \frac{5103 \sqrt{3} \pi^2}{200}}.$$
$$|| \varphi_{yy} || = \sqrt{\alpha^2 \left( \frac{5103 \pi^2}{40} + \frac{1568 \pi^4}{27} \right) + \beta^2 \left( -\frac{5103 \pi^2}{40} + \frac{1568 \pi^4}{27} \right) + 2 \alpha \beta \frac{5103 \sqrt{3} \pi^2}{40}  }.$$
We introduce the polar coordinates 
$$\cos(t) := \alpha, \quad \sin(t) := \beta.$$
The expressions simplify a bit,
$$|| \varphi_{xx} || = \sqrt{ \frac{1568 \pi^4}{27} - \cos(2t) \frac{15309 \pi^2}{200} - \sin(2t) \frac{15309 \sqrt{3} \pi^2}{200} }.$$
$$|| \varphi_{xy} || \sqrt{ \frac{1568 \pi^4}{81} - \cos(2t) \frac{5103\pi^2}{200} -\sin(2t) \frac{5103 \sqrt{3} \pi^2}{200}}.$$
$$|| \varphi_{yy} || = \sqrt{ \frac{1568\pi^4}{27} + \cos(2t) \frac{5103 \pi^2}{40} +  \sin(2t) \frac{5103 \sqrt{3} \pi^2}{40}  }.$$
Using these calculations and (\ref{eq:lvphi}), we estimate $|| L \varphi ||$ by determining the maximum of 
$$0.0013365 \sqrt{ \frac{1568 \pi^4}{27} - \cos(2t) \frac{15309 \pi^2}{200} - \sin(2t) \frac{15309 \sqrt{3} \pi^2}{200} }$$
$$+ 2.3148    \sqrt{ \frac{1568 \pi^4}{81} - \cos(2t) \frac{5103\pi^2}{200} -\sin(2t) \frac{5103 \sqrt{3} \pi^2}{200}}$$
$$+ 3.6512  \sqrt{ \frac{1568\pi^4}{27} + \cos(2t) \frac{5103 \pi^2}{40} +  \sin(2t) \frac{5103 \sqrt{3} \pi^2}{40}  },$$
for $0 \leq t < 2 \pi$.  We compute that the maximum is achieved when $t = \pi/6$, with the approximate value $416.269$, and in particular
$$||L \varphi|| \leq 416.27.$$
Substituting the estimate for $||L \varphi||$, the values of $\lambda_i$ for $i=1,2,3$, the estimate (\ref{eq:Lphi1}) for $||L \phi_1||$, and the estimate  (\ref{eq:g1}) for $||g||$ into the estimate for $||G||$ (\ref{eq:Gestf}), we arrive at our numerical estimate for $||G||$
\begin{equation} \label{eq:G1} ||G || \leq 4.7772.\end{equation} 
We use this to estimate $|| \n G||$ using the Poincar\'e inequality
$$|| \n G || \leq \sqrt{\lambda_3} ||G|| \leq \sqrt{\frac{64 \pi^2}{3}} 4.7772 \leq 69.32.$$
Recalling the estimate for $A$ (\ref{eq:a-1}), we have 
\begin{equation} \label{eq:a00} |A| \leq t \left( \frac{||g|| + t ||g||||G||}{\sqrt{1 - t^2 ||g||^2}} \right) \leq t(1.243).  \end{equation}  
Based on these estimates, we shall use the variational principle to estimate $\lambda_2 (t)$ from below.  Since $F - A \phi_1 = \varphi + tG$ is orthogonal to $\phi_1$, the variational principle for $\lambda_2$ gives 
\begin{equation} \label{l2vp} \lambda_2 \leq \frac{- \int (F- A\phi_1) \Delta_0 (F - A \phi_1)}{|| F - A \phi_1||^2} =  \frac{- \int (F- A\phi_1) (\Delta - t L) (F - A \phi_1)}{1+t^2 ||G||^2}. \end{equation} 
We compute the numerator to be 
$$\lambda_2 (t) + t \int F L F + \int A \phi_1 (\Delta - tL) (F-A \phi_1) + \int F(\Delta - tL) A \phi_1.$$
The last two terms are 
$$A \int \phi_1 \Delta_0 (\varphi + t G) + A \int F \Delta_0 \phi_1.$$
By definition of $\phi_1$ and integration by parts, these are 
$$= -A \lambda_1 \int \phi_1 (\varphi + tG) - A \lambda_1 \int F \phi_1.$$
The first integral vanishes since $\varphi$ and $G$ are orthogonal to $\phi_1$.  So, the numerator of (\ref{l2vp}) is 
$$- \int (F- A\phi_1) \Delta_0 (F - A \phi_1) = \lambda_2 (t) + t \int F L F - A \lambda_1 \int F \phi_1.$$
Thus, the variational principle for $\lambda_2$ implies
$$\lambda_2 \leq \frac{  \lambda_2 (t) + t \int F L F - A \lambda_1 \int F \phi_1}{1+t^2 ||G||^2},$$
which gives the estimate for $\lambda_2 (t)$,
$$\lambda_2 (t) \geq  \lambda_2 (1 + t^2 ||G||^2) - t \int FLF + A \lambda_1 \int F \phi_1.$$
This implies 
$$\lambda_2 (t) \geq \lambda_2 - t \int FLF + A \lambda_1 \int F \phi_1.$$
Substituting (\ref{eq:intGFA}), we have
\begin{equation}\label{eq:lambda2t1} \lambda_2 (t) \geq \lambda_2 - t \int FLF + A^2 \lambda_1 \geq \lambda_2 - t \int FLF. \end{equation}
By definition of $F = \varphi + A \phi_1 + tG$ and integration by parts, 
$$\int FLF = \int \varphi L \varphi + A^2 \int \phi_1 L \phi_1 + t^2 \int GLG + 2 A \int \varphi L \phi_1 + 2 t \int G L \varphi + 2At \int G L \phi_1.$$
This gives 
$$\lambda_2(t) \geq \lambda_2 - t \int \varphi (L_1|_{t=0}) \varphi - t \int \varphi (L - L_1|_{t=0}) \varphi $$
$$- t A^2 \int \phi_1 L \phi_1 - t^3 \int GLG - 2 At \int \varphi L \phi_1 - 2 t^2 \int G L \varphi - 2At^2 \int G L \phi_1.$$
Incorporating estimate (\ref{eq:o1t}) for $\lambda_1 (t)$, we have
$$\lambda_2 (t) - \lambda_1 (t) \geq \lambda_2 - \lambda_1 + t \int \phi_1 (L_1|_{t=0}) \phi_1 - t \int \varphi (L_1|_{t=0}) \varphi - |O_1(t^2)|$$
$$ - t \int \varphi (L - L_1|_{t=0}) \varphi $$
$$- t A^2 \int \phi_1 L \phi_1 - t^3 \int GLG - 2 At \int \varphi L \phi_1 - 2 t^2 \int G L \varphi - 2At^2 \int G L \phi_1.$$
Our calculations from the proof of Theorem \ref{th:lmin} and the estimate (\ref{eq:o1t}) of $|O_1(t^2)|$ imply 
$$\xi(T(t)) \geq \frac{64 \pi^2}{9} + (2.64) t  - t^2 (175.95) - t \int \varphi (L - L_1|_{t=0}) \varphi $$
$$- t A^2 \int \phi_1 L \phi_1 - t^3 \int GLG - 2 At \int \varphi L \phi_1 - 2 t^2 \int G L \varphi - 2At^2 \int G L \phi_1.$$
Recall the calculation
$$L - L_1|_{t=0} = \frac{4ta^2 \pa_x^2 - 4tb^2 \pa_y^2}{(\sqrt{3} + 2t)^2} + \left( \frac{-4\sqrt{3}}{(\sqrt{3} + 2tb)^2} + \frac{4}{\sqrt{3}} \right) \left( a \pa_x \pa_y + b \pa_y^2\right).$$
Thus, we have $\int \varphi (L - L_1|_{t=0}) \varphi =$
$$ -\frac{4ta^2}{(\sqrt{3} + 2t)^2} ||\varphi_x||^2 + \frac{4tb^2}{(\sqrt{3} + 2t)^2} ||\varphi_y||^2 + \left( \frac{4 \sqrt{3}}{(\sqrt{3} + 2tb)^2} - \frac{4}{\sqrt{3}} \right) \left( a \int \varphi_x \varphi_y + b ||\varphi_y||^2 \right).$$
We estimate using the Cauchy inequality, $||\varphi_x||$ and $||\varphi_y|| \leq ||\n \varphi||$ with $||\n \varphi||^2 = \lambda_2$, 
$$\int \varphi (L - L_1|_{t=0}) \varphi \leq \frac{4t}{(\sqrt{3} + 2t)^2} \lambda_2 +2 \left( \frac{4 \sqrt{3}}{(\sqrt{3} + 2tb)^2} - \frac{4}{\sqrt{3}} \right) \lambda_2.$$
Since we assume $t \leq 0.0004$, we have
$$\int \varphi (L - L_1|_{t=0}) \varphi \leq 0.58997.$$
We estimate the remaining terms using the Cauchy inequality, our estimates for $||L \varphi||$, $||L \phi_1||$, $G$, $||\n G||$, and the general estimate (\ref{eq:lcoeff}) for $L$, 
$$\left| \int \phi_1 L \phi_1 \right| \leq || L \phi_1 || \leq 86.194,$$
$$\left| \int GLG \right| \leq \frac{8 \sqrt{3} + 4t}{(\sqrt{3} - 2t)^2} || \n G ||^2 \leq 22664,$$
$$\left| \int \varphi L \phi_1 \right| \leq || L \phi_1 || \leq 86.194,$$
$$\left| \int G L \varphi \right| \leq ||G || ||L \varphi|| \leq 2009,$$
$$\left| \int G L \phi_1 \right| \leq ||G|| ||L \phi_1|| \leq 416,$$
since $||\phi_1|| = ||\varphi|| = 1$.  Substituting the estimate for $A$ gives our eventual estimate for the entire $O(t^2)$ error term, and we have
$$\xi(T(t)) \geq \frac{64 \pi^2}{9} + (2.64) t  - t^2 (175.95) - t (0.58997)$$
$$- t^3(1.243)^2 (86.194) - t^3(22664) - 2t^2(1.243)(86.194) - 2 t^2 (2009) $$
$$- 2t^3(1.243)(416).$$
This becomes
$$\xi(T(t)) \geq \frac{64 \pi^2}{9} + (2.05003) t $$
$$- t^2\left(175.95+ t*133.174 + t*22664 + 214.278+ 4018 + t*1034.18\right).$$
Thus, we compute the largest $t$ for which 
$$2.05003 > t \left(t*23831.4 + 4408.23 \right).$$
This is satisfied for any $t \leq 0.0004$.  
\qed

\section{Proof of Theorem \ref{th:emin}}
By our preceding results and continuity of the eigenvalues, we may now complete the proof of Theorem ~\ref{th:emin} by computing the first two eigenvalues of a large but finite number of triangles.   
\subsection{Continuity estimate}
The following calculation is based on the linear deformation theory at the beginning of Section 4.  We use $T(x,y)$ to denote a triangle with vertices $(0,0)$, $(1,0)$, and $(x,y)$, and we use $\lambda_i (x,y)$ to denote its $i^{th}$ Dirichlet eigenvalue, and $\xi(x,y)$ to denote its fundamental gap, 
$$\xi(x,y) = \lambda_2 (x,y) - \lambda_1 (x,y).$$  
If a triangle $T(x^*,y^*)$ satisfies 
$$(x^*-x)^2 + (y^*-y)^2 \leq t^2,$$
then by (\ref{eq:gen-ev-var}), 
\begin{equation} \label{mesh} \xi(x^*, y^*) \geq  \xi(x,y)  - \frac{2.4t}{y^2} \left( \lambda_2 (x,y) + \lambda_1 (x,y) \right). \end{equation}
Therefore, for each triangle $T(x,y)$ at which we compute numerically
$$\xi(x,y) > \frac{64 \pi^2}{9},$$
we may use (\ref{mesh}) to determine a neighborhood of triangles satisfying
$$\xi > \frac{64\pi^2}{9},$$
\em without numerically computing the eigenvalues of the triangles in this neighborhood.  \em  
Consequently, we have reduced the problem to numerically computing the fundamental gap of finitely many triangles and using the following algorithm.  

\subsection{Algorithm}
The main idea of the algorithm is to use the preceding calculations to compute, to sufficient numerical accuracy, the first two eigenvalues of \em a finite grid of triangles \em and use this grid together with the continuity estimate to demonstrate that the gap of all triangles lying outside the cases covered by Propositions 1 and 4 is strictly larger than that of the equilateral triangle.  In particular, it follows from Propositions 1 and 4, the invariance of the gap function under scaling and symmetry that we need only compute for those triangles with vertices 
$$(0,0), \quad (1,0) \quad \textrm{and} \quad (x,y),$$
such that the following inequalities hold. 

\begin{enumerate}
\item[i.1] $x^2 + y^2 \leq 1$, by invariance of the gap function under scaling.  
\item[i.2] $0.5 \leq x \leq 1$, by symmetry.  
\item[i.3] $0.005 \leq y \leq 1$, by Proposition 1. 
\item[i.4] $\sqrt{(x - 1/2)^2 + (y - \sqrt{3}/2)^2} > 0.0004$, by Proposition 4.  
\end{enumerate}

\subsubsection{The steps of the algorithm}  
We begin with the triangle whose vertices are $(0,0)$, $(1,0)$ and $(0.5, 0.005)$; this is step 0.  Next, in steps 1--2, we compute using (\ref{mesh}) the radius $t$ of the neighborhood around which the gap is strictly larger than $64\pi^2/9$.  We then increase the $x$-coordinate in step 3, and check that the inequalities i.1--i.4 hold.  If so, we repeat the calculations in steps 1--2 for the triangle whose third vertex is at the same height $y$ but has been translated in the positive $x$-direction (to the right).  We repeat steps 1--3 until the $x$ coordinate is large enough so that one of the inequalities i.1--i.4 fails; then we proceed to step 4.  In step 4, we return the $x$-coordinate to $0.5$ and increase the $y$-coordinate and check that the inequalities i.1--i.4 hold.  We then continue repeating steps 1--4.  
\begin{enumerate}
\item[0.] Initially, we define 
$$x_{0,j} := 0.5 \textrm{ for all $j$,} \quad \textrm{and} \quad y_0 := 0.005.$$  
\item[1.] At the $(i,j)^{th}$ iteration of the algorithm, where the first iteration of the algorithm is $(i,j) = (0,0)$, for the triangle with vertices 
$$(0,0), \quad (1,0) \quad \textrm{and} \quad (x_{i,j}, y_j)$$
we compute
\begin{enumerate}
\item[1.1] $\lambda_1$ and $\lambda_2$,
\item[1.2] $\xi = \lambda_2 - \lambda_1$ and  
\item[1.3] $A = \lambda_2 + \lambda_1$.
\end{enumerate}
\item[2.] We compute 
$$t_{i,j}' : = \left( \xi - \frac{64 \pi^2}{9} \right) \frac{y_j^2}{2.4 A}.$$
Let $n_{i,j}$ be the smallest $n \in \N$ such that the $10^{-n}$ digit in the decimal expansion of $t_{i,j}'$ is positive; let this digit be $d_{i,j}$.  Then, define 
$$t_{i,j} := 10^{-n_{i,j}} d_{i,j}.$$
The numerical method must be accurate up to the $10^{-n-1}$ decimal place; it then follows from the continuity estimate (\ref{mesh}) that 
for all triangles whose third vertex $(x,y)$ lies strictly within a neighborhood of radius $t_{i,j}$ about $(x_{i,j}, y_j)$, the gap function is strictly larger than $\frac{64 \pi^2}{9}$.  
\item[3.] We define 
$$x_{i+1,j} := x_{i,j} + t_{i,j},$$
and verify the following inequalities.  
\begin{enumerate}
\item[3.1] $(x_{i+1,j})^2 + (y_j)^2 \leq 1$.
\item[3.2] $\sqrt{(x_{i+1,j} - 0.5)^2 + \left( y_j - \frac{\sqrt{3}}{2} \right)^2} > 0.0004$.  
\end{enumerate}
If these inequalities are satisfied, repeat steps 1--3.  As soon as one of these inequalities is not satisfied,  proceed to step 4 below.  
\item[4.] Let $x_{0,j} = 0.5$, and for $j \geq 1$, define 
$$y_{j} = y_{j-1} + t_{0,j-1},$$
and verify the following inequalities. 
\begin{enumerate} 
\item[4.1] $(x_{0,j})^2 + (y_{j})^2 \leq 1$.
\item[4.2] $\sqrt{(x_{i,j} - 0.5)^2 + \left( y_j - \frac{\sqrt{3}}{2} \right)} > 0.0004$.
\end{enumerate}
If one of these inequalities is not satisfied, then the algorithm is complete.  If the inequalities are all satisfied, return to step 1 and repeat the algorithm.  
\end{enumerate}
\qed 

\subsection{The numerical methods}
The numerical computation of the eigenvalues were done by Timo Betcke using the Finite Element Method \em FreeFEM++ \em \cite{free}.  For efficiency, the calculations are made at each step but not stored, with the exception of $t_{0,j}$ which must be stored until it is replaced by $t_{0,j+1}$.  To demonstrate the behavior of the gap function numerically, Timo plotted the logarithm of the gap function in the figure below.  The grid points are parametrized so that each grid point corresponds to a triangle with vertices $(0,0)$, $(1,0)$ and $(x,y)$ where
$$x=1-\frac{\tau}{2} \quad \textrm{and} \quad y=\frac{\nu}{2}\sqrt{4-(2-\tau)^2}.$$ 
Hence, the equilateral triangle corresponds to $\nu=\tau=1$.   

\begin{figure}[ht] \label{gapplot} 
\includegraphics[width=11cm]{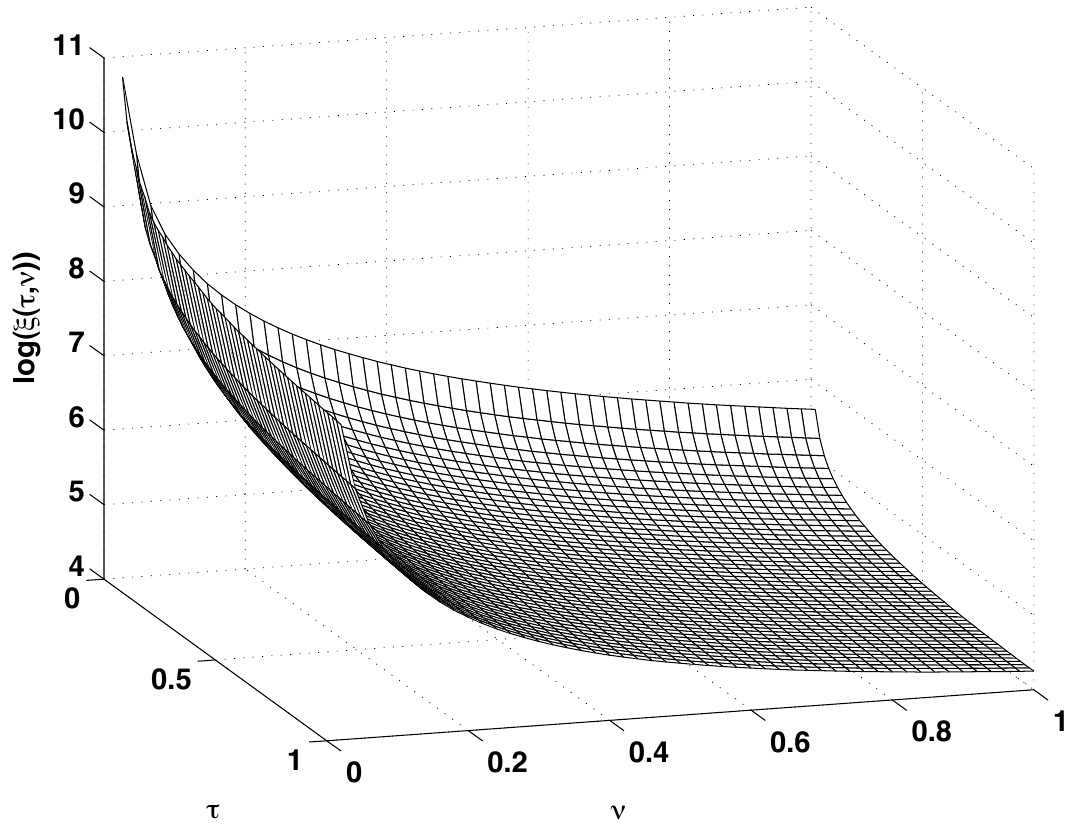} \caption{Plot of the logarithm of the gap function on the moduli space of triangles}
\end{figure}

\subsection{Concluding remarks}  
Based on the numerics, we make the following conjecture. 
\begin{conjecture} The logarithm of the gap function on the moduli space of triangles is a strictly convex function. 
\end{conjecture} 

Recently, Laugesen and Siudeja \cite{lau-s} proved an interesting related result.
\begin{theorem}[Laugesen-Siudeja] 
For any triangle of diameter $1$ with eigenvalues $\{ \Lambda_k\}_{k=1} ^{\infty}$, 
\begin{equation} \label{eq:ls} \sum_{k=1} ^n \Lambda_k \geq \sum_{k=1} ^n \lambda_k, \quad \forall \quad n \in \N, \end{equation} 
where $\{ \lambda_k \}_{k=1} ^{\infty}$ are the eigenvalues of the equilateral triangle. 
\end{theorem}

For $n=1$, (\ref{eq:ls}) is well known.  The case $n=2$ can be deduced from Theorem \ref{th:emin} as follows.  By Theorem \ref{th:emin} and (\ref{eq:ls}) with $n=1$, 
$$\Lambda_2 - \Lambda_1 \geq \lambda_2 - \lambda_1 \implies \Lambda_2 + \Lambda_1 \geq \lambda_2 - \lambda_1 + 2 \lambda_1 = \lambda_2 + \lambda_1.$$

The existence and identity of a gap-minimizing simplex is a challenging open problem.  Based on our results, we expect the following.  
\begin{conjecture}  Let $\M_n$ be the moduli space of all $n$-simplices with unit diameter.  For $n \geq 2$, the regular simplex defined by points $p_0, p_1, \ldots, p_n \in \R^n$ such that 
$$|p_i - p_j| = 1 \textrm{ for } 0 \leq i \neq j \leq n$$ 
uniquely minimizes the gap function on $\M_n$.  
\end{conjecture}

There are several difficulties to be addressed. A subtle problem is the behavior of the gap of a family of collapsing simplices when several directions collapse simultaneously.  Is it possible that competing collapsing directions may result in a gap which stays bounded or converges to that of the interval as simplices collapse?  Numerical calculations would provide insight into what one might expect;  combining classical techniques with modern computation may produce interesting new results.  
\\

We end this paper with a brief discussion of the similarities and differences between the behavior of the gap function on convex domains and the gap function restricted to the moduli space of $n$-simplices.  In the fundamental work of \cite{swyy} and subsequent papers \cite{yau3}, \cite{yau4} culminating in the proof of the fundamental gap conjecture \cite{ac}, the general method is to compare the eigenvalue estimate in higher dimensions to the eigenvalue estimate on a one dimensional manifold.  The minimum gap for all convex domains can be asymptotically approached by thin tubular domains, and the minimum is achieved in dimension one.  We pose the natural question:
\begin{quote}  
Is this minimum unique?  
\end{quote} 
More precisely, we make the following conjecture.
\begin{conjecture}
Let $\Omega \subset \R^n$ be a convex domain, and assume $n > 1$.  Then 
$$\xi(\Omega) > 3 \pi^2.$$
\end{conjecture} 
In the case of triangular domains, the gap function is \em uniquely minimized \em by the equilateral triangle.  It would be interesting to extend the beautiful works in the spirit of \cite{swyy} and \cite{ac} to compare the eigenvalue estimate in higher dimensions to the eigenvalue estimate in dimensions greater than one.  In particular, it would be interesting to compare the eigenvalue estimate to that on the equilateral triangle or other computable planar domains.

\end{document}